\documentclass[11pt]{amsart}
\usepackage{amsmath,amssymb,amsthm,latexsym}
\usepackage{graphicx}
\usepackage{stmaryrd}		
\usepackage{xcolor}
\usepackage{hyperref}
\hypersetup{colorlinks=true,    linkcolor=blue,
citecolor=red,       filecolor=BrickRed,       urlcolor=magenta
}
\usepackage{mathtools} 
\usepackage{bigints}
\usepackage{enumerate}
\usepackage{charter}
\usepackage{microtype}

\usepackage[capitalize]{cleveref}

\newcounter{alphathm}

\newtheorem{theorem}{Theorem}[section]
\newtheorem{corollary}[theorem]{Corollary}
\newtheorem{lemma}[theorem]{Lemma}
\newtheorem{proposition}[theorem]{Proposition}

\theoremstyle{definition}
\newtheorem{definition}[theorem]{Definition}

\newtheorem{remark}[theorem]{Remark}
\newtheorem{theorem*}[alphathm]{Theorem}
\numberwithin{equation}{section} 
\newcommand{\NN}{\mathbb N}
\newcommand{\ZZ}{\mathbb Z}

\newcommand{\RR}{\mathbb R}

\newcommand{\FF}{\mathbb F}
\newcommand{\Fp}{\FF_p}

\DeclareMathOperator{\spn}{span}

\newcommand{\bR}{R}
\newcommand{\F}{{\mathsf F}}

\newcommand{\N}{\mathbb N}

\newcommand{\colonequal}{\mathrel{\mathop:}=}

\newcommand{\bk}{k}
\newcommand{\bh}{\boldsymbol h}
\newcommand{\bt}{\boldsymbol t}
\newcommand{\br}{\boldsymbol r}
\newcommand{\bi}{\boldsymbol i}
\newcommand{\bj}{\boldsymbol j}
\newcommand{\by}{\boldsymbol y}

\newcommand{\bu}{\boldsymbol u}
\newcommand{\bv}{\boldsymbol v}
\newcommand{\brho}{\boldsymbol \rho}

\newcommand{\comp}{{\rm comp}}
\newcommand{\mB}{\mathcal B}
\newcommand{\res}{{\mathsf{res}}}

\newcommand{\Card}{{\rm Card}}
\newcommand{\NP}{{\rm NP}}
\newcommand{\GL}{{\rm GL}}

\newcommand{\softO}{\tilde{{O}}}
\usepackage{float}
\usepackage[noend]{algpseudocode}
\usepackage[shortlabels,inline]{enumitem}
\newfloat{algo}{tb}{lop}
\floatname{algo}{Algorithm}
\newenvironment{algoenv}[3][\linewidth]{
  \begin{minipage}{#1}%
  \flushleft
  \rule{\textwidth}{.08em}\vspace{-0.5\baselineskip}\smallskip
  \begin{description}[noitemsep]
  \item[Input] #2
  \item[Output] #3
  \end{description}
  \vspace{-0.7\baselineskip}
  \rule{\textwidth}{.05em}
  \begin{algorithmic}
}{\end{algorithmic}
\vspace{-.5\baselineskip}
\rule{\textwidth}{.08em}
\end{minipage}}

\def\pFqnoargs#1#2{{}_#1F_#2}
\def\pFq#1#2#3#4#5#6{\pFqnoargs{#1}{#2}\biggl(\begin{matrix}%
{#3}\kern.707em{#4}\\{#5}%
\end{matrix}\,\bigg|\,#6\biggr)}

\title[A sharper multivariate Christol's theorem]
{A sharper multivariate  Christol's theorem with \\  applications to diagonals and Hadamard products}
\date{}

\author{Boris Adamczewski}
\address{B. Adamczewski: Univ. Claude Bernard Lyon 1, CNRS UMR 5208, Institut Camille Jordan, 43 blvd. du 11 novembre 1918, F-69622 Villeurbanne cedex, France}
\email{boris.adamczewski@math.cnrs.fr}

\author{Alin Bostan}
\address{A. Bostan: Inria, Université Paris-Saclay, 1 rue Honoré d'Estienne d'Orves, 91120 Palaiseau, France}
\email{alin.bostan@inria.fr}

\author{Xavier Caruso}
\address{X. Caruso: CNRS, IMB, Universit\'e de Bordeaux, 351 cours de la Lib\'eration, 33405 Talence, France}
\email{xavier@caruso.ovh}


\thanks{Partially supported by the French grant DeRerumNatura (ANR-19-CE40-0018), and by
the French--Austrian project EAGLES (ANR-22-CE91-0007 \& FWF I6130-N)}

\begin{document}
\maketitle

\begin{abstract}
We provide a new proof of the multivariate version of Christol's theorem about algebraic power
series with coefficients in finite fields, as well as of its extension to perfect ground fields of
positive characteristic obtained independently by Denef and Lipshitz, Sharif and Woodcok, and
Harase.
Our proof is elementary, effective, and allows for much sharper estimates. We discuss various
applications of such estimates, in particular to a problem raised by Deligne concerning the
algebraicity degree of reductions modulo $p$ of diagonals of multivariate algebraic power series
with integer coefficients. 
\end{abstract}

\setcounter{tocdepth}{1}
\tableofcontents

\section{Introduction}\label{sec: intro}

Rational and algebraic power series play an important role in various areas of mathematics, and
especially in number theory and combinatorics.
There are two fundamental results concerning rationality of power series in one variable.  
The first one is that, given an arbitrary field $k$, a power series $f(t)=\sum_{n=0}^\infty
a(n)t^n\in k[[t]]$ is rational if and only if its coefficient sequence $a(n)$ satisfies a linear 
recurrence with coefficients in $k$.
The second one is the famous Skolem-Mahler-Lech theorem stating that, when $k$ is a field of
characteristic zero, the zero set
\[
\mathcal Z(f) \coloneqq\{n\in\mathbb N : a(n)=0\}
\]
of a rational power series $f\in k[[t]]$ is a \emph{periodic} set, that is the union of a finite
set and of finitely many arithmetic progressions.
When $k$ has characteristic zero, it seems difficult to obtain similar results for \emph{algebraic}
power series, \emph{i.e.} power series $f \in k[[t]]$ for which there exists a nonzero bivariate
polynomial $P \in k[t,y]$ such that $P(t,f) = 0$. 
Although it is known that the coefficient sequences of univariate algebraic power series do
satisfy linear recurrences with polynomial coefficients, a characterization of such sequences is
still lacking. On the other hand, proving that the corresponding zero sets are periodic remains a
challenging open problem (cf.~\cite[p.\ 176]{Za09}). 
The situation in several variables is worse and not much is known or even conjectured about zero
sets of multivariate rational power series; this is very unfortunate, as they encode interesting
Diophantine problems.

In a short but influential paper, Furstenberg \cite{Fu67} observed for the first time that
algebraic power series over ground fields of positive characteristic have a very particular
structure. For instance, when $k=\mathbb F_q$ is a finite field, he proved that the ring of
algebraic power series is closed under the Hadamard product and that $\sum_{n=0}^\infty a(n)t^n\in
\mathbb F_q[[t]]$ is algebraic over $\mathbb F_q(t)$ if and only if $\sum_{n : a(n)=a}t^n$ is
algebraic for every $a\in \mathbb F_q$; these two properties do not hold in characteristic zero.
Later, Christol \cite{Ch79} elaborated on Furstenberg's approach and proved the following beautiful
result.

\begin{theorem*}[Christol]
\label{thm:christol}
A power series $\sum_{n=0}^\infty a(n)t^n\in \mathbb F_q[[t]]$ is algebraic over $\mathbb
F_q(t)$ if and only if its coefficient sequence $a(n)$ can be generated by a finite
$q$-automaton.
\end{theorem*}

We refer the reader to \cite[Chapters 4 and 5]{AS03} for the notions of finite automata and
automatic sequences. A different proof of Christol's theorem was given in \cite{CKMR}, from which
Salon ~\cite{Sa87,Sa86} also derived a natural extension to multivariate power series. What makes
Christol's theorem deep and fascinating is that it establishes an intimate connection between two
important objects coming from seemingly unrelated areas.

When the ground field $k$ is an arbitrary field of positive characteristic, the characterization of
algebraic power series in terms of finite automata is somewhat lost but there are still important
related results. Furstenberg \cite{Fu67} proved that the diagonal of a multivariate rational power
series remains algebraic; this result was generalized by Deligne \cite{De84} to diagonals of
multivariate algebraic power series. A related result is that the ring of multivariate algebraic
power series is closed under the Hadamard product (cf.\ \cite{DL,SW,Ha88}). More recently, Derksen
\cite{De07} found an appropriate version of the Skolem-Mahler-Lech theorem involving finite
automata. Derksen's theorem was then generalized to zero sets of arbitrary multivariate algebraic
power series by Adamczewski and Bell \cite{AB12}.

It turns out that all the previously mentioned results (in positive characteristic) can be proved
by using some splitting process associated with the Frobenius map (cf.\ Equality \eqref{eq:
basis1}).
Let $\bk$ be a perfect field of characteristic $p>0$. 
Then the Frobenius endomorphism~$\F$, that maps $x$ to $x^p$, is an automorphism of $\bk$. 
Let $\bt = (t_1,\ldots,t_n)$ be indeterminates, and let~$K$ denote the field of fractions of the
ring of power series $k[[\bt]]\coloneqq \bk[[t_1,\ldots,t_n]]$.
The Frobenius map $\F$ extends naturally to $K$ as an injective homomorphism. 
We let~$K^{\langle p\rangle}$ denote the image of~$K$ by $\F$, so that $\F$ defines an isomorphism
between $K$ and~$K^{\langle p\rangle}$.
Then $K$ is a $K^{\langle p\rangle}$-vector space of dimension $p^n$, a basis being given by all
the products of the form $\bt^{\br}\coloneqq t_1^{r_1}\cdots t_n^{r_n}$, with $\br
\coloneqq(r_1,\ldots,r_n)\in \{0,\ldots,p-1\}^n$.
Thus, every $f\in K$ has a unique expansion of the form 
\begin{equation}
\label{eq: basis1}
	f= \sum_{\br \in \{0,\ldots,p-1\}^n}\bt^{\br} f_{\br}\,.
\end{equation}
For every $\br\in \{0,\ldots,p-1\}^n$, the \emph{section operator} $S_{\br}$ is defined by 
\begin{equation}
\label{eq: section1}
S_{\br}(f)\coloneqq \F^{-1}(f_{\br})\,.
\end{equation}
Section operators are semilinear maps from $K$ into itself. 
For a power series $f \coloneqq\sum_{\bi \in \mathbb N^n}a(\bi)\bt^{\bi}$ $\in
k[[\bt]]$, we have
\[
S_{\br}(f) = \sum_{\stackrel{\bi \in \mathbb N^n}{\bi=(i_1,\dots ,i_n)}}a(pi_1+r_1,\ldots,pi_n+r_n)^{1/p}\bt^{\bi}\in k[[\bt]]] \,.
\]
We let $\Omega_n$ denote the monoid generated by all section operators under composition.

At the end of the 1980s, Denef and Lipshitz \cite{DL}, Sharif and Woodcock \cite{SW} and Harase
\cite{Ha88} obtained independently the following nice characterization of multivariate algebraic
power series in terms of section operators.

\begin{theorem*} 
\label{thm:B}
Let $k$ be a perfect field of characteristic $p$. A power series $f\in \bk[[\bt]]$ is algebraic
over $\bk(\bt)$ if and only if there exists a finite-dimensional $k$-vector space $W\subset K$
containing $f$ and invariant under the action of $\Omega_n$. 
\end{theorem*}

\begin{remark}
When $k=\mathbb F_q$, the putative vector space $W$ has to be finite and its existence is thus equivalent to the finiteness of the orbit of $f$ under $\Omega_n$. 
By a classical result of Eilenberg, the latter property is itself equivalent to the fact that the coefficient sequence of  $f$ is $q$-automatic,  
so one recovers the multivariate extension of Christol's theorem.   
\end{remark}

In this paper, we investigate the problem of finding a sharp quantitative version of Theorem~\ref{thm:B},
\emph{i.e.}\ finding a $k$-vector space $W$ having the smallest possible dimension in terms of the
``complexity'' of the algebraic power series~$f$. 

The particular case of multivariate rational power series can be treated in a satisfactory way.
Indeed, if $f(\bt)=A(\bt)/B(\bt)\in k[[\bt]]$ for some $A,B\in k[\bt]$, whose degree in $t_i$ is at
most $h_i$, then it is easy to prove that the $k$-vector space
\[
W  \coloneqq \left\{ \frac{P(\bt)}{B(\bt)} : P\in k[\bt], \; \deg_{t_i} P \leq h_i, 1\leq i\leq n\right\}
\] 
contains $f$ and is invariant under $\Omega_n$. Similarly, if the total degree of both $A$ and $B$ is at most $h$, then the same conclusion applies to 
\[
W' \coloneqq \left\{ \frac{P(\bt)}{B(\bt)} : P\in k[\bt], \; \deg_{\bt} P \leq h \right\}\,.
\]
Furthermore, $W$ and $W'$ have dimensions $(h_1+1)\cdots (h_n+1)$ and $\binom{n+h}{n}$, respectively. 

On the other hand, the case of algebraic irrational multivariate power series is known to be more
difficult and the known estimates are much weaker.
To summarize roughly, the main aim of this paper is to develop a unified method, which is able to
deal with arbitrary multivariate algebraic power series as if they were rational. 
In this direction, our main results are \cref{thm: main} and \cref{cor: main}. The proof of \cref{thm: main} 
follows a new approach recently initiated in \cite{BCCD} for
the case $n=1$. Its main feature is that it combines all the advantages of
the different methods used so far to study (part of) this problem (cf.\
\cite{Ch79,CKMR,Sa86,Sa87,DL,SW,Ha88,Ha89,AB12,AB13,Br17,AY}). 
Indeed, our method is elementary, as
general as possible in the sense that it applies to arbitrary multivariate algebraic power series
over arbitrary fields of characteristic~$p$, and it provides much sharper quantitative estimates. 
The methods used to date are briefly discussed in  \cref{subsec: comparison}.

\subsection{Statement of our main result}\label{sec: mainresult} 

The ``complexity'' of an algebraic power series $f\in k[[\bt]]$ is classically measured by its
\emph{degree} and its \emph{height}.
The degree of $f$ is the minimal degree in~$y$ of a nonzero polynomial $A(\bt,y)\in k[\bt,y]$ such
that $A(\bt,f)=0$. It is also equal to $[k(\bt)(f):k(\bt)]$, the degree of the field extension
$k(\bt)(f)$ of~$k(\bt)$.
For the height, we have two natural choices, since we can consider either the \emph{partial
height} or the \emph{total height}. We say that $f$ has partial height $\bh =(h_1,\ldots,h_n)$
if, for all $i$, $h_i$ is the minimal degree in $t_i$ of a nonzero polynomial $A(\bt,y)\in
k[\bt,y]$ such that $A(\bt,f)=0$, while the total height of~$f$ is the minimal total degree in
$\bt$ of such polynomials $A(\bt,y)$.
In fact, a finer way to measure the complexity of multivariate polynomials, and hence of algebraic
power series, is to consider their Newton polytopes. We recall that the \emph{Newton polytope} (or,
\emph{Newton polyhedron}) $\NP(A)$ of a multivariate polynomial
 \[A  \coloneqq \sum_{\substack{\bi \in \NN^n \\ j \in \NN}}
  a_{\bi,j} \bt^{\bi} y^j \in \bk[\bt,y]\]
is defined as the convex hull in $\RR^{n+1}$ of the tuples $(\bi,j)$ with $a_{\bi,j} \neq 0$.
An important result is that, for a generic polynomial $A$, the number of integer points in the
interior of $\NP(A)$ is equal to the (geometric) genus of the hypersurface associated with $A$.
This has been proved by Baker for plane curves~\cite{Baker1893}, by Hodge for
surfaces~\cite{Hodge29} and by Khovanskii for arbitrary hypersurfaces~\cite{Khovanskii78}. 
As a result, \cref{thm: main} has a geometric flavor, as does the result obtained by
Bridy~\cite{Br17} in the case where $k$ is a finite field and $n=1$ (cf.\ \cref{sec: christol} for
more details).

Keeping the previous notation, our main result reads as follows. 

\begin{theorem} \label{thm: main}
Let $\bk$ be a perfect field of characteristic $p$.
Let $A(\bt,y)$ be a nonzero polynomial in $\bk[\bt,y]$ and let $f\in \bk[[\bt]]$ satisfy the
algebraic relation $A(\bt,f)=0$.
Set
\[ C  \coloneqq \NP(A)+ (-1,0]^{n+1} \,. \]
Then there exists a $\bk$-vector space $W\subset K$ of dimension at most
\begin{equation*}\label{eq: NP} \Card(C \cap \NN^{n+1})
\end{equation*}
that contains $f$ and that is closed under the action of $\Omega_n$.
\end{theorem}

\begin{remark}\label{rem: thm}
The plus sign in the definition of $C$ refers to the Minkowski sum. In Theorems~\ref{thm:B} and
\ref{thm: main}, the field $\bk$ must be perfect for the section operators to be well-defined, but
this is not a real limitation. Indeed, replacing an arbitrary field of characteristic $p$ by its
perfect closure does not affect our results (cf.\ \cref{sec: descente}). \end{remark}

When measuring complexity of algebraic power series in terms of degree and height, our main
result can be translated as follows.

\begin{corollary}\label{cor: main}
We keep the notation of \cref{thm: main}. Let us further assume that $\deg_y(A)\leq d$,
$\deg_{\bt}(A)\leq h$, and $\deg_{t_i}(A)\leq h_i$ for all $1\leq i\leq n$.
Then, there exists a $\bk$-vector space $W\subset K$ of dimension at most 
\begin{equation*}
N \coloneqq (d+1) \cdot \min \left\{  \prod_{i=1}^n(h_i+1) , \binom{n+h}{n} \right\}
\end{equation*}
that contains $f$ and that is closed under the action of $\Omega_n$.  
\end{corollary}

\begin{remark}
In the rest of the paper, we express our results mostly in terms of degree and height rather than
in terms of Newton polytopes, which leads to somewhat less precise bounds (compare, for instance,
\cref{thm: main} and \cref{cor: main}).
This is the case for \cref{thm: christol_eff,thm: modp,thm: hadamard,thm: HL}. 
The reason for this choice is that we think it could be more meaningful for some readers. However,
there is no difficulty in deducing from our arguments and \cref{thm: main} more precise
results involving Newton polytopes.
\end{remark}

\subsection{Motivation}

\cref{thm: main} is motivated by various applications, in particular to the four following
problems:

\medskip
\begin{itemize}

\item[{\rm (i)}] Given an algebraic power series 
\[f(\bt)\coloneqq\sum_{\bi \in\mathbb N^n}a(\bi)\bt^{\bi}\in\mathbb F_q[[\bt]]\,,\] 
find upper bounds for the minimal number of states required for a $q$-automaton in order to
generate the sequence $(a(\bi))_{\bi\in\mathbb N^n}$.

\item[{\rm (ii)}] Given an algebraic power series  
\[ f(\bt)\coloneqq\sum_{\bi \in\mathbb N^n}a(\bi)\bt^{\bi}\in\mathbb Z[[\bt]]\,,\] 
find upper bounds for the degree of algebraicity over $\mathbb F_p(t)$ of the reduction modulo $p$
of the diagonal of $f$, that is
\[
\Delta(f)_{\mid p}  \coloneqq \sum_{i=0}^{\infty} (a(i,\ldots,i)\bmod p) t^i\in\mathbb F_p[[t]]\,.
\]

\item[{\rm (iii)}] Given two multivariate algebraic power series over an arbitrary field of
characteristic $p$, find upper bounds for the degree of algebraicity of their Hadamard product.

\item[{\rm (iv)}] Given an algebraic power series  
\[f(\bt)\coloneqq\sum_{\bi \in\mathbb N^n}a(\bi)\bt^{\bi}\in\mathbb \FF_q[[\bt]]\]
(encoded by its minimal polynomial over $\FF_q(\bt)$ and a large enough enclosure which guarantees
uniqueness) and given a multi-index $\bi \in \NN^n$, find fast algorithms to compute the
coefficient $a(\bi)$. 

\end{itemize}

\medskip

For each of them, our results strongly improve and/or generalize previously known results (cf.\
Sections~\ref{sec: christol}--\ref{sec: algo}). For example, in connection with Problem (ii), we
significantly improve the main upper bound obtained by Adamczewski and Bell \cite{AB13} for the
algebraicity degree of reductions modulo~$p$ of diagonals of algebraic power series with integer
coefficients, thus fully answering a question raised in 1984 by Deligne \cite{De84} (see \cref{thm:
modp}).

\subsection{Comparison of previous methods}\label{subsec: comparison}

There are essentially three different approaches to prove results in the vein of Theorems~\ref{thm:christol} and
\ref{thm: main}. We briefly recall them together with their main advantages and shortcomings.

\medskip

(1) The most classical approach, initiated in \cite{CKMR} and used in
\cite{Sa86,Sa87,DL,SW,Ha88,Ha89,AB12}, is based on the fact that algebraic power series are roots
of polynomials of a certain type called \emph{Ore polynomials}, \emph{i.e.}\ they satisfy algebraic
equations of the form
\[ a_0f+a_1f^p+\cdots+a_df^{p^d}=0 \,, \]
where $a_0,\ldots,a_d\in k[\bt]$, not all zero. The main advantage of this approach is that it is
elementary and general in the sense that it applies to arbitrary multivariate algebraic power
series over arbitrary fields of positive characteristic. Its main deficiency is that it provides
poor bounds for the dimension of the $k$-vector space $W$, namely bounds of the form $p^A$,
where~$A$ can be made explicit and depends polynomially on the parameters $d$ and $h$ (or $d$ and $h_1,\ldots,h_n$),
that is on the degree and the height of $f$.

\medskip

(2) The second approach is based on the so-called \emph{rationalization process}. This means that
one expresses the algebraic power series $f$ as the diagonal of a rational function with more
variables.
Its main advantage is that it provides bounds on the dimension of the $k$-vector space $W$ that do
not depend on $p$. It is based on Furstenberg's formula \cite[Proposition 2]{Fu67} which is 
elementary.
The main drawback is that Furstenberg's formula only applies to algebraic power series satisfying
some specific polynomial relations. For algebraic power series in one variable, one can overcome
this difficulty by using resultant techniques (see, for example, \cite{AY}). For multivariate power
series, one can use inductively these resultant methods as in \cite{AB13}, but the worst case
scenario leads to really huge bounds.
The same approach is also used in \cite{DL}, where the rationalization process is ensured by a
nonelementary and ineffective argument (cf.\ \cite[Remark~3.1]{AB13}).

\medskip

(3) The last approach consists in using a geometric setting, considering the projective curve (or
hypersurface) associated with the algebraic power series $f$.
It was used by Deligne \cite{De84} in order to reprove Furstenberg's theorem on diagonals, and more
recently by Bridy \cite{Br17} to prove a strong quantitative version of Christol's theorem in
dimension one. The main advantage of the geometric approach is that it seems to provide tight bounds
for the dimension of the $k$-vector space $W$, which is indeed the case in \cite{Br17}.
The main shortcomings of this method are that it is not elementary, since it requires tools from
algebraic geometry, and that it does not seem to work as well in the multivariate setting.

\subsection{Organization of the paper}

\cref{sec: main} is devoted to the proof of~\cref{thm: main}. 
In \cref{sec: descente}, we prove a result ensuring that one can consider, without any loss in
our bounds, algebraic power series over arbitrary ground fields of positive characteristic (instead
of perfect fields as assumed in \cref{thm: main}).
This result is used in \cref{sec: diag,sec: hadamard}.  
In~\cref{sec: christol}, we discuss Problem~(i) and prove, in the multivariate setting, results
similar to those obtained by Bridy~\cite{Br17} for univariate algebraic power series. \cref{sec:
diag} is devoted to Problem~(ii). We consider diagonals of algebraic power series in several
variables over arbitrary fields of characteristic~$p$, and obtain a general upper bound for their
algebraicity degree, which significantly improves the main known result in this direction obtained
by Adamczewski and Bell \cite{AB13}.
In \cref{sec: hadamard}, we discuss Problem~(iii). We obtain the first simply exponential bound
(with respect to the characteristic $p$ of the ground field) for the degree of algebraicity of the
Hadamard product of two multivariate algebraic power series. We also prove a similar result for the
Hurwitz and the Lamperti products of two algebraic power series in one variable.
This considerably improves the doubly exponential bounds which follow from~\cite{DL,SW,Ha88} and
that were made explicit by Harase~\cite{Ha89}.
Finally, in \cref{sec: algo}, we consider some algorithmic consequences of \cref{thm: main} to
Problem~(iv). In particular, we show that the $M$-th coefficient of the diagonal of an algebraic
power series $f\in \mathbb{F}_p[[t_1,\ldots,t_n]]$ can be computed using a number of operations in
$\mathbb{F}_p$ which is logarithmic in $M$ and almost linear in $p^{n+1}$. Again, this improves
significantly upon previously known results.

\medskip

\noindent Throughout the paper, we let $\mathbb N \coloneqq \{0,1,\ldots\}$ denote the set of nonnegative integers. 

\section{Proof of \cref{thm: main}}\label{sec: main}

This section is devoted to the proof of our main result.

\subsection{A variant of Furstenberg's formula}

Let $K$ be a field and $T$ be an indeterminate. The residue map $\res$ is defined from the field of 
Laurent series $K((T))$ to $K$ by setting
\[
\res \Bigg(\sum_{n\geq \nu}a_nT^n \Bigg) \coloneqq a_{-1} \,.
\]
Given a polynomial $A \in K[y]$, we let $A_y$ denote its derivative with respect to~$y$. The
following key lemma is a slight reformulation of \cite[Lemma~2.3]{BCCD}, itself inspired by
\cite[Proposition~2]{Fu67}. We provide a proof for the reader's convenience.

\begin{lemma}
\label{lem: furstenberg}
Let $K$ be a field, $f \in K$ and $A(y)\in K[y]$.
Assume that $A(f)=0$ and $A_y(f)\not=0$.  
Then 
\[
\res \left(\frac{P(f+T)}{A(f+T)}\right)=\frac{P(f)}{A_y(f)} \,,
\]
for all $P\in K[y]$. 
\end{lemma}

\begin{proof}
By assumption, the polynomial $A(f+T)\in K[T]$ has a simple root at $T=0$. There thus exists a
polynomial $Q(T)\in K[T]$ such that
\[
A(f+T)=T\cdot Q(T) \;\;\; \mbox{ and } \;\;\; Q(0)\not=0\,.
\]
Taking the logarithmic derivative with respect to $T$ yields the equality
\[
\frac{A_y(f+T)}{A(f+T)}= \frac{1}{T} + \frac{Q'(T)}{Q(T)} \, \cdot
\]
Setting $g(T) \coloneqq P(f+T)/A_y(f+T)$, we find that $g$ belongs to $K[[T]]$, as $A_y(f+T)$ does
not vanish at $T=0$. Furthermore, one has
\[
\frac{P(f+T)}{A(f+T)}= \frac{g(T)}{T} + \frac{g(T)Q'(T)}{Q(T)} \,\cdot
\]
Similarly, $g(T)Q'(T)/Q(T)$ belongs to $K[[T]]$ since $Q(0)\not=0$. 
It follows that 
\[
\res \left(\frac{P(f+T)}{A(f+T)}\right)= g(0) = \frac{P(f)}{A_y(f)}\, ,
\]
as claimed. 
\end{proof}

\subsection{Frobenius and section operators}

Let $\bk$ be a perfect field of characteristic~$p$.  
We keep on with the notation introduced in~\cref{sec: intro}. 
We let $\F$ be the Frobenius map. 
We consider indeterminates $t_1,\ldots,t_n$ and write $\bt = (t_1, \ldots, t_n)$.
We set $K_0\coloneqq k(\bt)$, $R \coloneqq k[[\bt]]$ and $K  \coloneqq {\rm Frac}(R)$.  
For every $\br\in \{0,\ldots,p-1\}^n$, the section operator $S_{\br}$ is the semilinear operator from $K$ into itself defined as in~\eqref{eq: section1}.

\subsubsection{Section operators on $K((T))$}

Let us consider a new indeterminate~$T$. Then $\F$ extends to an injective homomorphism from
$K((T))$ into itself.
We let $K((T))^{\langle p\rangle}$ denote the image of $K((T))$ by $\F$. As previously, $K((T))$ is
a $K((T))^{\langle p\rangle}$-vector space of dimension $p^{n+1}$, a basis being given by all the
products of the form $\bt^{\br}T^s$ with $\br \in \{0,\ldots,p-1\}^n$ and $0\leq s\leq p-1$.
However, it will be more convenient for our purpose to replace this standard basis by a more
appropriate one (depending on a given $f\in R$). We proceed in the same way as in
\cite[Lemma~2.4]{BCCD}.

\begin{lemma}
\label{lem: basis}
For any $f\in R$, the family 
\[
\mB_f \colonequal \big\{\bt^{\br}(f+T)^s : \br,s\in \{0,\ldots,p-1\}^{n+1}\big\}
\]
is a basis of $K((T))$ as a $K((T))^{\langle p\rangle}$-vector space. 
\end{lemma}

\begin{proof}
First, we observe that $\mB_f$ is a generating family. Indeed, we can obtain $\bt^{\br}T^s$ as a
linear combination of $\bt^{\br}(f+T)^i$, $0\leq i\leq s$.
To see this, it is enough to invert the matrix
\[
\left(
\begin{matrix}
1&0&\cdots &\cdots& \cdots& 0 \\
f&1&0 &\cdots& \cdots& 0\\
f^2&2f&1&0&\cdots&   0\\
\vdots &\vdots& \vdots&\ddots&\vdots& \vdots\\
\vdots &\vdots& \vdots&\vdots&\ddots& \vdots\\
f^s&sf^{s-1} &\frac{s(s-1)}{2} f^{s-2}&\cdots&sf& 1
\end{matrix}\right).
\]
Since $\mB_f$ has the same cardinality as the basis $\{\bt^{\br}T^s : \br,s \in
\{0,\ldots,p{-}1\}^{n+1}\}$, it is also a basis of $K((T))$. 
\end{proof}

It follows that, given $f\in R$,  every $x\in K((T))$ has a unique expansion of the form 
\begin{equation}
\label{eq: basis2}
x= \sum_{\br \in \{0,\ldots,p-1\}^n} \bt^{\br} \sum_{s=0}^{p-1} (f+T)^s x_{f,\br,s}\,,
\end{equation}
with $x_{f,\br,s} \in K((T))^{\langle p\rangle}$.
For every $\br\in \{0,\ldots,p-1\}^n$ and $s\in\{0,\ldots,p-1\}$, we define the \emph{section 
operator} $S_{f,\br,s}$, from $K((T))$ into itself, by 
\begin{equation}
\label{eq: section2}
S_{f,\br,s}(x)\coloneqq \F^{-1}(x_{f,\br,s})\,.
\end{equation} 
We observe that 
$S_{f,\br,s}(xy^p) = S_{f,\br,s}(x\F(y)) = S_{f,\br,s}(x)y$, 
for all $x,y\in K((T))$, all $\br\in \{0,\ldots,p-1\}^n$, and all $s\in\{0,\ldots,p-1\}$.

\subsubsection{Section operators and residues}

The next result is the key observation of our proof.
Roughly speaking, it shows some compatibility between taking residues
at $f(t)$ and residues at $0$.
It corresponds to a multivariate extension of \cite[Proposition 2.5]{BCCD}.

\begin{proposition}
\label{prop: res}
For any $f\in R$ and $\br \in \{0, \ldots, p{-}1\}^n$, 
the following commutation relation holds over~$K((T))$:
\[S_{\br} \circ \res =  \res \circ S_{f,\br,p-1}\, .\]
\end{proposition}

\begin{proof}
Let $x\in K((T))$.  By \eqref{eq: basis2}, we have  
\[
x= \sum_{\br \in \{0,\ldots,p-1\}^n} \bt^{\br} 
	\sum_{s=0}^{p-1} (f+T)^s \, \F(S_{f,\br,s}(x)) \,.
\]
Hence 
\begin{eqnarray*}
\res (x)  &=& \sum_{\br \in \{0,\ldots,p-1\}^n}\bt^{\br} \, 
	\sum_{s=0}^{p-1} \res \Big( (f+T)^s \, \F(S_{f,\br,s}(x)) \Big)\\
&=& \sum_{\br \in \{0,\ldots,p-1\}^n}\bt^{\br} \, \F\big(\res \left( S_{f,\br,p-1}(x)\big) \right) \,.
\end{eqnarray*}
By \eqref{eq: basis1} and \eqref{eq: section1}, we obtain 
\[
S_{\br} \circ \res (x)= \res \circ  S_{f,\br,p-1}(x) \,,
\]
as wanted. 
\end{proof}

\subsection{Proof of \cref{thm: main}}

We keep on with the previous notation and the notation of \cref{sec: intro}.
Let $E(\bt,y)\in \bk[\bt,y]$ denote the minimal polynomial of $f$
over $k(t)$, normalized so that its coefficients are globally coprime.

\begin{lemma}
\label{lem:separable}
The power series $f$ is a simple root of $E(\bt, y)$.
\end{lemma}

\begin{proof}
It is enough to show that $E(\bt, y)$ is separable with respect to
the variable~$y$.
Since it is defined as a minimal polynomial, this further reduces
to prove that $E(\bt, y)$ is not of the form $F(\bt, y^p)$ for some
polynomial $F(\bt,z)\in \bk[\bt,z]$.
We assume by contradiction that this occurs, and we write
\[
F(\bt, z) = a_0(\bt) + a_1(\bt) z + \cdots + a_m(\bt) z^m
\]
with $a_i(\bt) \in \bk[\bt]$ and $a_m(\bt) \neq 0$.
Let $\br \in \{0, \ldots, p{-}1\}^n$. Applying the section operator 
$S_{\br}$ to the identity $F(\bt, f^p) = 0$, we obtain
\[
S_{\br}\big(a_0(\bt)\big) + 
S_{\br}\big(a_1(\bt)\big) f + \cdots + 
S_{\br}\big(a_m(\bt)\big) f^m = 0\,.
\]
Moreover, since $a_m(\bt)$ is nonzero, there must exist $\br$ for
which $S_{\br}(a_m(\bt))$ does not vanish either.
For this particular $\br$, we then get a polynomial annihilating
$f$ with $y$-degree less than the $y$-degree of~$E$. This
contradicts the minimality of~$E$.
\end{proof}

Let $E_y$ be the partial derivative of
$E$ with respect to $y$. \cref{lem:separable} ensures that
$E_y(\bt,f)\not=0$.
Besides, given that $A$ annihilates $f$, it must be a multiple of $E$, \emph{i.e.}\ we can write $A
= E \cdot F$ for some polynomial $F\in \bk[\bt,y]$.
Let $J$ be the interval $(-1,0]$ and set $C' \colonequal \NP(E) + J^{n+1}$.

We claim that the $\bk$-vector space 
\[
W \coloneqq \left\{ \frac{P(\bt,f)}{E_y(\bt,f)} : 
P\in \bk[\bt,y], \; \NP(P) \subset C' \right\} \subset K 
\]
contains $f$ and is invariant under the action of $\Omega_n$.
The fact that $f \in W$ follows from the observation that $\NP(y E_y) \subset \NP(E) \subset C'$.
We now consider a tuple $\br \in \{0, 1, \ldots, p{-}1\}^n$ together with a polynomial $P \in
\bk[\bt,y]$ whose Newton polytope is a subset of $C'$.
We set $R \colonequal P \cdot E^{p-1}$ and let $Q \in \bk[\bt,y]$ be defined by 
\begin{equation}
\label{eq:formulaQ}
Q(\bt,f+T)  \coloneqq S_{f,\br,p-1}(R(\bt,f+T))\in K \,.
\end{equation}
Combining~\cref{lem: furstenberg} and \cref{prop: res}, we obtain:
\begin{eqnarray} \label{Sr_P_Q}
S_{\br}\left( \frac{P(\bt,f)}{E_y(\bt,f)}\right)& =& 
	S_{\br} \circ \res  \left( \frac{P(\bt, f+T)}{E(\bt,f+T)}\right)\\\nonumber
&=&  \res  \circ  S_{f,\br,p-1} \left( \frac{P(\bt, f+T)}{E(\bt,f+T)}\right)\\ \nonumber
&=& \res   \left( \frac{Q(\bt, f+T)}{E(\bt,f+T)}\right)\\\nonumber
&=& \frac{Q(\bt,f)}{E_y(\bt,f)} \,\cdot
\end{eqnarray}

To establish our claim, it just remains to prove that $\NP(Q) \subset C'$.
We recall the following standard fact about Newton polytopes. The formation of Newton polytopes is
compatible with products: given $A, B \in \bk[\bt,y]$, we have the relation
\[\NP(AB) = \NP(A) + \NP(B).\]
From this property, we derive that 
\[\NP(R) \subset (p{-}1){\cdot}\NP(E) + C' = p{\cdot}\NP(E) + J^{n+1}.\]
Let $(\bi,j)$ be a tuple of exponents that belongs to the support of $Q$, \emph{i.e.}\ for which
the coefficient in $Q$ in front of $\bt^{\bi} y^j$ is nonzero.
It follows from the definition of $S_{f,\br,p-1}$ that $(p\bi + \br, pj + p-1)$ must lie in 
$\NP(R)$. 
Dividing by $p$ and writing $I \coloneqq (-\frac 1 p, 0]$, we obtain that 
\[\textstyle
\left(\bi + \frac 1 p \br, \, j + \frac{p-1} p\right) \in \NP(E) + 
I^{n+1}\,, \]
so that 
\[\textstyle
(\bi, j) \in \NP(E) + I^{n+1} + 
\left\{\left(-\frac 1 p \br, -\frac{p-1} p\right)\right\}
\subset \NP(E) + J^{n+1} = C'\,.\]
Finally, we conclude that $\NP(Q) \subset C'$, as wanted.

Clearly $W$ is spanned by the fractions of the form 
$\bt^{\bi} f^j/E_y(\bt,f)$ with $(\bi, j) \in C' \cap \NN^{n+1}$.
Hence its dimension is upper bounded by the cardinality of this set. 
We observe moreover that $C = \NP(F) + C'$. Since $F$ is nonzero, its Newton polytope $\NP(F)$ meets
$\NN^{n+1}$. Hence $C$ contains a translate of $C'$ by an element with nonnegative integral
coefficients.
Consequently the cardinality of $C\cap \NN^{n+1}$ is at least that of $C'\cap \NN^{n+1}$, and we
conclude that \[\dim_{\bk} W \leq \Card(C'\cap \NN^{n+1})\leq \Card(C\cap \NN^{n+1})\,,\] as wanted.

\begin{remark}
\label{rem: main}
The proof above actually implies the following statement, which is a little more precise than
\cref{thm: main} and can be useful is some cases.
For a nonnegative integer $m$, let $J_m$ be the interval $(-1,\, -1{+}p^{-m}]$ and define:
\begin{align*}
C'_m & \coloneqq \NP(E) + (J_0^n \times J_m) \,, \\
W_m & \coloneqq \left\{ \frac{P(\bt,f)}{E_y(\bt,f)} : 
P\in \bk[\bt,y], \; \NP(P) \subset C'_m \right\} \,.
\end{align*}
The $W_m$'s form a nonincreasing sequence of $\bk$-vector spaces and the action of $\Omega_n$ sends
$W_m$ to $W_{m+1}$.
In particular, it stabilizes the intersection of the $W_m$'s.
However, it is not true that $f$ belongs to $W_m$ for all $m$: in full generality, it only lies 
in~$W_0$.

If we set $\Omega_n^{+}{\cdot} f \coloneqq \{ S_{\br_1}\circ\cdots \circ S_{\br_t}(f) : t\geq 1\}$,
we obtain that $\Omega_n{\cdot} f= \{f \}\cup \Omega_n^{+}{\cdot} f$, while $\Omega_n^{+}{\cdot}f$
is contained in $W_1$, which can be strictly smaller than $W_0$.
For example, if we assume that $\deg_y(A)\leq d$, $\deg_{\bt}(A)\leq h$, and $\deg_{t_i}(A)\leq
h_i$ for all $1\leq i\leq n$, then
\[\dim_k W_0 \leq (d+1)\cdot \min \left\{  \prod_{i=1}^n(h_i+1), \binom{n+h}{n} \right\} \,, \]  
whereas 
\[\dim_k  W_1 \leq d\cdot \min \left\{ \prod_{i=1}^n(h_i+1), \binom{n+h}{n} \right\}\,.\]
\end{remark}

\section{From perfect to arbitrary
fields of positive characteristic}\label{sec: descente}

In \cref{thm: main}, the ground field $\bk$ must be perfect for the section operators to be
well-defined. If $k$ is not perfect, one can always replace $k$ by its perfect closure and then
apply \cref{thm: main}. In this section, we prove a result which ensures that passing to the
perfect closure does not affect the bounds we obtain in \cref{sec: diag,sec: hadamard}.

Recall that if $k$ is an arbitrary field of characteristic~$p$, then adjoining to~$k$ all the
$p^r$-th roots ($r \geq 1$) of all the elements of~$k$ yields a perfect field; it is called the
perfect closure of $k$ and we will denote it by~$k_p$.

\begin{proposition}\label{prop: descente}
Let $k$ be an arbitrary field of characteristic $p$ and let $k_p$ be its perfect closure.
Let $f\in k[[\bt]]$ be algebraic over $k(\bt)$. 
Then, $[k(\bt)(f):k(\bt)]=[k_{p}(\bt)(f):k_{p}(\bt)]$. 
\end{proposition}

\cref{prop: descente} is a direct consequence of the following lemma.

\begin{lemma}\label{lem: descente} 
Let $k_0$ be a field, $k_1$ be an extension of $k_0$, $\bt = (t_1,\ldots,t_n)$ 
be indeterminates, and $f_1(\bt),\ldots,f_r(\bt)\in {k_0}[[\bt]]$.
If the power series $f_1,\ldots,f_r$ are linearly dependent over the field $k_1(\bt)$, then they
are linearly dependent over the field~$k_0(\bt)$.

\end{lemma}

\medskip

\begin{proof} By assumption, there exist polynomials $A_i(\bt)\in  k_1[\bt]$, not all zero, such 
that 
\begin{equation}\label{eq: rel}
\sum_{i=1}^rA_i(\bt)f_i(\bt)=0\,.
\end{equation}
Set  
\[
A_i(\bt) \coloneqq\sum_{{\bf j}\in \mathcal S_i}b_{i,\bf j}{\bf t}^{\bf j}\,
\]
where we let $\mathcal S_i\subset \mathbb N^n$ denote the support of $A_i$, and 
\[
f_i(\bt) \coloneqq\sum_{\bj \in\mathbb N^n} a_{i,\bf j}{\bf t}^{\bf j}\in {\bf k}_0[[\bt]]\,.
\]
Let $(e_\ell)_{\ell \in L}$ be a basis of $k_1$, seen as a ${k}_0$-vector space. 
Thus, there exist  some $c_{i,{\bf j},\ell}\in k_0$ such that 
\[
b_{i,\bf j} \coloneqq\sum_{\ell\in L} c_{i,{\bf j},\ell} e_\ell\,.
\]
Then, Equality \ref{eq: rel} implies that for all  ${\bf k}\in\mathbb N^n$, one has 
\begin{equation*}
\sum_{i=1}^r \sum_{{\bf j} \in \mathcal S_i} b_{i,{\bf j}} a_{i,{\bf k} - {\bf j}} =0 \,,
\end{equation*}
and hence 
\begin{equation*}
\sum_{i=1}^r \sum_{{\bf j} \in \mathcal S_i}\sum_{\ell\in L} c_{i,{\bf j},\ell} e_\ell a_{i,{\bf k} 
- {\bf j}} =\sum_{\ell\in L} \left( \sum_{i=1}^r \sum_{{\bf j} \in \mathcal S_i}c_{i,{\bf j},\ell} 
a_{i,{\bf k} - {\bf j}}\right)e_\ell=0 \,.
\end{equation*}
Since the $e_\ell$'s are linearly independent over $k_0$, we obtain 
\begin{equation}\label{eq: cjk}
 \sum_{i=1}^r \sum_{{\bf j} \in \mathcal S_i}c_{i,{\bf j},\ell} a_{i,{\bf k} - {\bf j}}=0
\end{equation}
for all $\ell\in L$ and all ${\bf k}\in\mathbb N^n$.
Setting $A_{i,\ell}(\bt) \coloneqq \sum_{{\bf j}\in\mathcal S_i}c_{i,{\bf j},\ell} {\bf t}^{\bf
j}\in {k_0}[\bt]$, Equality \eqref{eq: cjk} implies that
\begin{equation}\label{eq: aik}
\sum_{i=1}^rA_{i,\ell}(\bt)f_i(\bt)=0, \quad \text{for all} \;  \ell \in L\,.
\end{equation}
Since the polynomials $A_i$ are not all zero, the coefficients $b_{i,\bf j}$ are not all zero, and
the same is also true for the coefficients $c_{i,{\bf j},\ell}$.
Hence, there exists an index $\ell$ such that the polynomials $A_{i,\ell}$, $1\leq i\leq r$, are
not all zero.
We thus deduce from \eqref{eq: aik} that $f_1,\ldots,f_r$ are linearly dependent over ${k}_0(\bt)$,
as wanted.
\end{proof}


\section{State complexity in Christol's theorem}\label{sec: christol}

Given an integer $q\geq 2$, a multidimensional sequence ${\bf a}=(a(\bi))_{\bi\in\mathbb N^n}$ with
values in a finite set is said to be {\em $q$-automatic} if $a(\bi)$ is a finite-state function of
the base-$q$ expansions of the entries of $\bi$. This means that there exists a deterministic
finite automaton taking the base-$q$ expansion of each entry of~$\bi$ as input, and producing the
symbol $a(\bi)$ as output. For a formal definition, we refer the reader to \cite{AB_chapter}.

For the rest of this section, we let $q$ denote a prime power.  
The multivariate extension of Christol's theorem can be stated as follows. 
It is usually proved by following the approach initiated in \cite{CKMR} for the case $n=1$ (cf.\
\cite{Sa86,Sa87,DL,SW,Ha88}).

\begin{theorem*}
Let $f(\bt)\coloneqq\sum_{\bi \in\mathbb N^n}a(\bi)\bt^{\bi}\in \mathbb F_q[[\bt]]$.
Then $f$ is algebraic over $\mathbb F_q(\bt)$ if and only if the sequence ${\bf a} \coloneqq
(a(\bi))_{\bi\in\mathbb N^n}$ is $q$-automatic. 
\end{theorem*}

On each side, there is a natural way to measure the complexity of the corresponding objects, as
described below.
A natural problem is then to study the interplay between the complexity of 
the algebraic power series $f$ and that of its sequence of coefficients $\bf a$.

\subsection{Two notions of complexity}

As already mentioned in \cref{sec: intro}, the complexity of an algebraic power series $f\in
\mathbb F_q[[\bt]]$ can be measured by its degree $d$ and either its partial height $\bh\coloneqq
(h_1,\ldots,h_n)$ or its total height $h$.
We recall that $d\coloneqq [\mathbb F_q(\bt)(f):\mathbb F_q(\bt)]$, and, for all $i$, $1\leq i\leq
n$, $h_i$ (resp.\ $h$) is the minimal degree in the variable $t_i$ (resp.\ the minimal total degree
in~$\bt$) of a nonzero polynomial $A(\bt,y)$ such that $A(\bt,f)=0$.
These two notions of height are the same when $n=1$.

The complexity of a $q$-automatic sequence $\bf a$ is measured by its \emph{state complexity}. We
let $\stackrel{\longleftarrow}{\comp_q}({\bf a})$ denote the number of states in a minimal finite
automaton generating $\bf a$ in {\em reverse} reading, by which we mean that the input $\bi$ is
read starting from the least significant digits. In a similar way, we let $
\stackrel{\longrightarrow}{\comp_q}({\bf a})$ denote the state complexity of $\bf a$ with respect
to {\em direct} reading.
In general, $\stackrel{\longleftarrow}{\comp_q}({\bf a})$ and
$\stackrel{\longrightarrow}{\comp_q}({\bf a})$ behave quite differently and are only related by the
inequalities $\stackrel{\longleftarrow}{\comp_q}({\bf a}) \leq
q^{\stackrel{\longrightarrow}{\comp_q}({\bf a})}$ and $\stackrel{\longrightarrow}{\comp_q}({\bf a})
\leq q^{\stackrel{\longleftarrow}{\comp_q}({\bf a})}$. 
These estimates are derived from classical bounds for converting a nondeterministic finite
automaton into a deterministic one (see, for example, \cite[Chapter 4]{AS03}).

\subsection{Previous bounds on the state complexity}

Let us first recall that if ${\bf a}$ is generated by a $q$-automaton with at most $m$ states in
reverse reading, it is not difficult to show that the associated power series $f$ has degree $d\leq
q^m-1$ and total height $h\leq mq^{m}$ (this follows, for instance, from the proofs of
\cite[Propositions 5.1 and 5.2]{AB13}). Furthermore, it seems that these bounds cannot be
significantly improved in general.

Bounds in the other direction are more challenging. 
The approach based on Ore's polynomials, initiated in \cite{CKMR} and pursued in
\cite{Sa86,Sa87,DL,SW,Ha88,Ha89,AB12}, leads to bounds of the form
\[\stackrel{\longleftarrow}{\comp_q}({\bf a})\leq q^{Aq^{B}}\, ,\] where $A$ and $B$ are polynomial
functions of the parameters $d,h_1,\ldots,h_n$, that can be made explicit. 
The common feature of these bounds is that they have a doubly exponential nature (with respect to
the size $q$ of the ground field).

By contrast, when $f$ is a rational function (\emph{i.e.}\ $d=1$), one can easily obtain the bound
$q^N$, where $N\coloneqq \min \{(h_1+1)\cdots (h_n+1), \binom{n+h}{n}\}$, and thus get rid of the
double exponential.
This follows from \cref{prop: invariant1} using the vector spaces $W$ and $W'$
introduced in \cref{sec: intro}, and suggests that the previous bounds are artificially
large.
In a more recent paper, Bridy \cite{Br17} drastically improved on these doubly exponential bounds
in the case $n=1$. More precisely, he proved that 
\begin{equation}\label{eq: bridy}
\stackrel{\longleftarrow}{\comp_q}({\bf a})  \leq (1+o(1))q^{h+d+g-1} \,,
\end{equation}
where $g$ is the genus of the projective curve associated with $f$, and where the $o(1)$ term tends
to $0$ for large values of any of $q,h,d$, or $g$.
By Riemann's inequality, which gives $g\leq (h-1)(d-1)$, he deduced that 
\begin{equation}\label{eq: bridy2}
\stackrel{\longleftarrow}{\comp_q}({\bf a})  \leq (1+o(1))q^{hd} \,.
\end{equation}
He also proved that 
\begin{equation}\label{eq: bridy3}
\stackrel{\longrightarrow}{\comp_q}({\bf a})  \leq q^{(h+1)d} \,.
\end{equation}

Bridy's approach is based on a new proof of Christol's theorem in the context of algebraic
geometry, due to Speyer (see his blog post untitled \emph{Christol's theorem and the Cartier
operator}\footnote{Available at {\url{https://sbseminar.wordpress.com/2010/02/11}}.}).
Speyer's argument is elegant, connecting finite automata with the geometry of curves. However, the
price to pay to get \eqref{eq: bridy} is that some classical but nonelementary background from
algebraic geometry is needed: the Riemann-Roch theorem, the existence and the basic properties of
the \emph{Cartier} operator acting on the space of K\"ahler differentials of the function field
associated with $f$, along with asymptotic bounds for the Landau function. Also, this geometric
method does not seem to generalize easily to higher dimension. In an unpublished note, Adamczewski
and Yassawi \cite{AY} showed how a slightly weaker bound can be obtained in an elementary way using
diagonals, as in the original proof of Christol's theorem \cite{Ch79}, and resultant techniques.
However, the use of resultants makes the proof somewhat tedious.

\subsection{A simply exponential bound in all dimensions}

As a consequence of \cref{thm: main}, we obtain the following bound valid in any dimension. 

\begin{theorem}\label{thm: christol_eff}
Let $f(\bt)\coloneqq\sum_{\bi \in \mathbb N^n}a(\bi)\bt^i \in \mathbb F_q[[\bt]]$ be algebraic over
$\mathbb F_q(\bt)$ with degree~$d$, total height $h$, and partial height $\bh\coloneqq(h_1,\ldots,h_n)$.
Set 
\[
N \coloneqq d\cdot \min\left\{  \prod_{i=1}^n(h_i+1), \binom{n+h}{n}\right\} 
\]
and 
\[
 N' \coloneqq (d+1) \cdot \min\left\{  \prod_{i=1}^n(h_i+1), \binom{n+h}{n}\right\} \,.
 \]
Then 
\[\stackrel{\longleftarrow}{\comp_q}({\bf a})
\leq 1+ q^{N}  
\quad \mbox{ and } \quad
  \stackrel{\longrightarrow}{\comp_q}({\bf a})
\leq q^{N'}\,.\] 
\end{theorem}

For $n=1$ we obtain $\stackrel{\longleftarrow}{\comp_p}({\bf a})\leq 1+ p^{(h+1)d}$ and
$\stackrel{\longrightarrow}{\comp_q}({\bf a})\leq q^{(h+1)(d+1)}$; these estimates are close to
Bridy's bounds \eqref{eq: bridy2} and \eqref{eq: bridy3}.
Note that this could be pushed a little by considering separately the orbit of $f$ under the
section operator~$S_{\bf 0}$.
In fact, this is precisely how Bridy proceeds to get~\eqref{eq: bridy2}. Let $C \coloneqq
\mbox{NP}(A)+(-1,0]^2$ be defined as in \cref{thm: main} (in the case $n=1$) and let $g_{A}$
denote the number of integer points in the interior of $\mbox{NP}(A)$.
Generically, we have that $g=g_{A}$. 
To simplify the exposition, the bounds given in \cref{thm: christol_eff} are obtained by
overapproximating $C\cap \mathbb N^2$ (for instance, by $(h_1+1)(d+1)$ for the second one). Using
$g_{A}$ instead, we would obtain bounds with the same flavor as~\eqref{eq: bridy}.
\cref{thm: christol_eff} is new in dimension $n\geq 2$, where only doubly exponential bounds
were available until now (cf.\ \cite{Ha89,FKdM,AB12} and the discussion in~\cite{Br17}).

The proof of \cref{thm: christol_eff} is derived from \cref{thm: main} and the following
result.

\begin{proposition}
\label{prop: invariant1}
Let $f(\bt)\coloneqq\sum_{\bi\in\mathbb N^n}a(\bi)\bt^{\bi}\in \mathbb F_q[[\bt]]$.  
Assume that there exists a $\mathbb F_q$-vector space $W\subset \mathbb F_q((\bt))$ 
of dimension $m$ containing $f$ and  invariant under the action of $\Omega_n$. 
Then 
\begin{equation}
\label{eq: kerbis}
\max\{\stackrel{\longrightarrow}{\comp_q}({\bf a}), \stackrel{\longleftarrow}{\comp_q}({\bf a}) \} \leq q^m\,.
\end{equation}
Furthermore, we have
\begin{equation}
\label{eq: ker}
\stackrel{\longleftarrow}{\comp_q}({\bf a}) = \vert \Omega_n\cdot f\vert \,. 
\end{equation}
\end{proposition}

In the case $n=1$, Inequality~\eqref{eq: kerbis} is a rephrasing of \cite[Proposition 2.4]{Br17},
while Equality~\eqref{eq: ker} is a rephrasing of a classical result of Eilenberg which asserts
that $\stackrel{\longleftarrow}{\comp_q}({\bf a})$ is equal to the cardinality of the $q$-kernel of
the sequence $\bf a$. Both results extend straightforwardly to arbitrary positive integers $n$.

\begin{proof}[Proof of \cref{thm: christol_eff}]
The upper bound for $\stackrel{\longleftarrow}{\comp_q}({\bf a})$ follows from~\cref{thm:
main} and Equation \ref{eq: kerbis}. The upper bound for
$\stackrel{\longrightarrow}{\comp_q}({\bf a})$ is a direct consequence of \cref{rem: main}
and Equation~\eqref{eq: ker}. 
\end{proof}

\section{Diagonals}\label{sec: diag}


Given a field $k$ and a multivariate power series 
\[
f(t_1,\ldots,t_n) \coloneqq  \sum_{(i_1,\ldots,i_n)\in\mathbb N^n} a(i_1,\ldots,i_n)t_1^{i_1}\cdots t_n^{i_n} \in k[[t_1,\ldots,t_n]]\,,
\]
the \emph{diagonal} of $f$ is defined as the univariate power series 
\[
\Delta(f)(t) \coloneqq  \sum_{i=0}^{+\infty} a(i,\ldots,i) t^i \in k[[t]]\, .
\] 
When $k$ is a number field, diagonals of algebraic functions form a remarkable class of power
series: they satisfy linear differential equations of Picard-Fuchs type, they belong to the class of
Siegel's $G$-functions, and they are constantly reoccurring in enumerative combinatorics.
Furthermore, diagonalization is related to integration and, in general, the diagonal of an
algebraic power series is transcendental over $k(t)$. For more details, we refer to the
survey~\cite{Ch15}.

By contrast, Furstenberg \cite{Fu67} proved that if $k$ has characteristic $p$ and $f$ is a
rational power series, then $\Delta(f)$ is algebraic over $k(t)$. In \cite{De84}, Deligne
generalized this result to diagonals of algebraic power series. Then Harase \cite{Ha88}, Sharif and
Woodcock \cite{SW}, Denef and Lipshitz \cite{DL}, as well as Salon \cite{Sa87,Sa86} (in some
particular case) independently reproved Deligne's theorem.

Combining \cref{thm: main} with Propositions~5.1 and~5.2 of \cite{AB13}, we readily obtain an
effective version of Deligne's theorem: given an algebraic power series $f \in \bk[[\bt]]$ with
degree $d$ and total heights $h$, the diagonal $\Delta(f)$ has degree at most $p^N$ and height at
most $N p^N$, where $N$ is explicitly given by
\[
  N \coloneqq (d+1)\cdot \binom{n+h}{n}\,.
\]
In this section, we will prove a further refinement of this result, which can be formulated as
follows.

\begin{theorem}
\label{thm:diag}
Let $k$ be an arbitrary field of characteristic $p$. 
Let $f \in \bk[[\bt]]$ be an 
algebraic power series with degree $d$, total height $h$, and partial
height $\bh=(h_1,\ldots,h_n)$. Set
\begin{equation*}
N \coloneqq (d+1)\cdot \min \left \{ 
\prod_{i=1}^n (h_i+1) - 
\prod_{i=1}^n h_i
, 
\binom{n+h}{n}-\binom{h}{n} \right\}\,.
\end{equation*}
Then, there exist $c_0, c_1, \ldots, c_N \in k[t]$, not all zero, such that 
\[c_0 \cdot \Delta(f) +  c_1 \cdot \Delta(f)^p + \cdots + 
c_N \cdot \Delta(f)^{p^N} = 0\, .\]
In particular, $\Delta(f)$ has degree at most $p^N-1$. 
\end{theorem}

\subsection{Reduction of diagonals modulo primes}

Given a prime number $p$ and a power series $f(\bt)\coloneqq \sum_{\bi \in\mathbb N^n} a(\bi)
\bt^{\bi} \in \mathbb Z[[\bt]]$, we let $f_{\vert p}$ denote the reduction of $f$ modulo $p$, that
is
\[
f_{\vert p}(\bt) \coloneqq  \sum_{\bi \in\mathbb N^n} (a(\bi)\bmod{p})  \bt^{\bi} \in  \mathbb F_p[[\bt]] \,.
\] 
Deligne \cite{De84} made the following nice observation: since diagonalization and reduction modulo
$p$ commute, that is $\Delta(f)_{\vert p} = \Delta(f_{\vert p})$, if $f(\bt)\in \mathbb Z[[\bt]]$
is algebraic over $\mathbb Q(\bt)$, then $\Delta(f)_{\vert p}$ is algebraic over $\mathbb F_p(t)$
for almost all prime~$p$.
Hence, it is natural to ask how the ``complexity'' of the algebraic function $\Delta(f)_{\vert p}$
may increase when $p$ runs along the primes. When $\Delta(f)$ is transcendental, van der Poorten
\cite{vdP90} conjectured that the degree of $\Delta(f)_{\vert p}$ cannot remain bounded
independently of $p$. On the other hand, Deligne \cite{De84} suggested that the degree of
$\Delta(f)_{\vert p}$ should grow at most polynomially in $p$.

Using the vector spaces $W$ and $W'$ introduced in \cref{sec: intro}, we can deduce the
polynomial bound $p^N$, with $N\coloneqq \min \{ (h_1+1)\cdots(h_n+1), \binom{n+h}{n}\}$, for $f$
a multivariate rational power series with total height $h$ and partial height
$\bh=(h_1,\ldots,h_n)$.
The case where $f$ is not rational is much more challenging. 
Deligne \cite{De84} obtained a first result in this direction by proving that if $f(t_1,t_2)\in
\mathbb Z[[t_1,t_2]]$ is algebraic, then, for all but finitely many primes $p$, $\Delta(f)_{\vert
p}$ is of degree at most $Ap^B$, where $A$ and $B$ do not depend on $p$ but only on certain
geometric quantities associated with~$f$.
On the other hand, the works of Harase \cite{Ha88,Ha89}, Sharif and Woodcock \cite{SW}, and
Adamczewski and Bell \cite{AB12} lead to doubly exponential bounds (\emph{i.e.}\ of the form
$p^{p^{M}}$).
The first general polynomial bound (\emph{i.e.}\ of the form $p^{A}$) was obtained by Adamczewski
and Bell in \cite{AB13}. They provide an effective $A$ that depends only on the degree and the
total height of $f$. However, when $f$ has degree $d>1$, the value of $A$ becomes huge due to a
recursive procedure involving resultants.
For  instance, even for $n=2$, the estimate for $A$ is of the form 
\[d^{4^{(h^2d^6)^{d^{4^{hd^2}}}}}\]
and the length of the exponential tower increases at least linearly with~$n$. 
It is also possible to deduce from the work of Denef and Lipshitz \cite{DL} the existence of such
a polynomial bound, but with an ineffective constant~$A$.

\cref{thm:diag} readily implies the following result, which quite significantly improves the
previous known bounds.

\begin{theorem}
\label{thm: modp}
Let $f \in \ZZ[[\bt]]$ be an algebraic power series with degree $d$, total height~$h$, and partial
height $\bh=(h_1,\ldots,h_n)$. Set
\begin{equation}
\label{eq:N}
N \coloneqq (d+1)\cdot \min \left \{ 
\prod_{i=1}^n (h_i+1) - 
\prod_{i=1}^n h_i
, 
\binom{n+h}{n}-\binom{h}{n} \right\}\,.
\end{equation}
Then, for all prime numbers $p$, 
$\Delta(f)_{|p}$ has degree at most $p^N - 1$ over $\mathbb
F_p(t)$. 
\end{theorem}

\begin{remark} \label{rem:GLN}
Not only \cref{thm:diag} gives a nice bound on the degree of $\Delta(f)_{|p}$, but it also shows
that $\Delta(f)_{|p}$ is annihilated by an Ore polynomial of bounded $p$-degree.
This additional feature implies that the Galois conjugates of $\Delta(f)_{|p}$ are all contained in
an $\Fp$-vector space of dimension $N$ and eventually that the Galois group of $\Delta(f)$
(\emph{i.e.}\ the Galois group of the extension of $k(t)$ generated by $\Delta(f)$ and all its
Galois conjugates) canonically embeds, up to conjugacy, into $\GL_N(\Fp)$.
This observation allows for asking more precise questions about the uniformity with respect to~$p$.
For example, one may wonder if the Galois groups of $\Delta(f)_{|p}$ all come by reduction modulo
$p$ from a unique group (or maybe a finite number of groups) defined in characteristic zero.
\end{remark}

\begin{remark}
Using the same arguments as in \cite[p.~967]{AB13}, we could also prove a more general statement 
than \cref{thm: modp} by replacing the ring~$\mathbb Z$ with a 
number field and consider reductions modulo prime ideals.
In fact, we could even consider the case where $f$ has coefficients in an arbitrary field of
characteristic zero (see \cite[Theorem~1.4]{AB13}).
Note that, beyond diagonals of algebraic power series, there are other interesting families of
$G$-functions in $\mathbb Q[[t]]$ whose reductions modulo $p$ are algebraic (cf.\ \cite{VM21}).
Furthermore, algebraicity modulo $p$ turns out to be useful to prove transcendence and algebraic
independence results for power series in characteristic zero (cf.\ \cite{SW89,AGBS,AB13,ABD,VM23}).
\end{remark}

\subsection{Generalized diagonals}

In what follows, we consider a slight generalization of the diagonalization process.
Let $k$ be a perfect field of characteristic~$p$ and let $\bt \coloneqq  (t_1,\ldots, t_n)$ be a tuple of indeterminates. We set $K_0 \coloneqq k(\bt)$, $R \coloneqq k[[\bt]]$, and we let 
$K$ denote the field of fractions of $R$. 
Let $G$ be a subgroup of $\ZZ^n$ such that the quotient $\ZZ^n/G$ has
no torsion. We let $K_{0,G}$ be the subfield of $K_0$ generated by $k$ and
by the monomials $\bt^{\bi}$ with $\bi \in G$. Similarly, we define
$\bR_G$ as the $\bk$-subalgebra of $\bR$ consisting of series of
the form $\sum_{\bi \in G} a(\bi) \bt^{\bi}$.
Given that $G$ is abstractly isomorphic to $\ZZ^m$ for some integer
$m \leq n$, the rings $K_G$ and $\bR_G$ are respectively isomorphic to 
$\bk(x_1,\ldots, x_m)$ and $\bk[[x_1, \ldots, x_m]]$.

\begin{definition}
We keep the previous notation. 
The \emph{$G$-diagonal} is the operator defined by 
\[\begin{array}{rcl}
\Delta_G : \qquad
\bR & \longrightarrow & \bR_G \medskip \\
\displaystyle \sum_{\bi \in \NN^n} a(\bi) \bt^{\bi}
& \mapsto &
\displaystyle \sum_{\bi \in G} a(\bi) \bt^{\bi}
\end{array}\]
with the convention that $a(\bi) = 0$ when $\bi \not\in \NN^n$.
\end{definition}

When $G$ is the subgroup generated by $(1, \ldots, 1)$, the ring
$R_G$ is isomorphic to $k[[t]]$ \emph{via} the map $t_1 \cdots t_n \mapsto t$ and the diagonal operator $\Delta_G$ is the usual diagonal operator $\Delta$.
However the general construction $\Delta_G$ is more flexible and
allows in particular for partial diagonals: letting $G$ be the
subgroup generated by $(1, \ldots, 1)$ and by $e_i = (0, \ldots,0, 1, 0, \ldots, 0)$ (with $1$ in $i$-th position) for 
$i \in \{1, \ldots, m\}$, we obtain that $\bR_G \simeq k[[t_1, \ldots, t_m, x]]$
and 
\[\Delta_G\left(\sum_{i \in \NN^n} a(\bi) \bt^{\bi}\right) \,\, =
\sum_{\substack{(i_1, \ldots, i_m)\in \NN^m \\ n \in \NN}}
a(i_1, \ldots, i_m, n, \ldots, n)\: t_1^{i_1} \cdots t_m^{i_m} x^n\,.\]
In general, one can check that $\Delta_G$ is $K_{0,G}$-linear.

\begin{theorem}
\label{thm:gendiag}
Let $k$ be an arbitrary field of characteristic $p$,  
$G$ be a subgroup of $\ZZ^n$ such that $\ZZ^n/G$ has no
torsion, let $G_{\RR}$ be the subvector space of $\RR^n$ 
generated by $G$, and let 
$\pi_G : \RR^{n+1} \to (\RR^n/G_{\RR}) \times \RR$ denote the 
canonical projection. 
Let $A(\bt,y)\in \bk[\bt,y]$ and let $f\in \bk[[\bt]]$ 
satisfying the algebraic relation $A(\bt,f)=0$. Let $C$ be the
convex subset of $\RR^{n+1}$ defined by
\[
C  \coloneqq \NP(A) + \big(G_\RR \times (-1,0]\big)\, .
\]
Then, there exist $c_0, c_1, \ldots, c_N \in K_{0,G}$, not all zero, such that 
\begin{equation}
\label{eq:Ore}
c_0 \cdot \Delta_G(f) +  c_1 \cdot \Delta_G(f)^p + \cdots + 
c_N \cdot \Delta_G(f)^{p^N} = 0 \,,
\end{equation}
where $N  \coloneqq \Card \big(\pi_G(C \cap \NN^{n+1})\big)$.
\end{theorem}

\begin{proof} 
We first observe that, by \cref{prop: descente}, we can replace without any loss of
generality the field $k$ by the perfect closure of the subfield of $k$ generated over~$\mathbb F_p$
by the coefficients of $f$. Hence, we can assume that $k$ is perfect.

Let $E \in \bk(\bt, y)$ be the minimal polynomial of $f$ and $E_y$ be the derivative of $E$ with
respect to $y$.
Set $J  \coloneqq (-1,0]$ and 
\[
C'  \coloneqq \NP(E) + \big(G_\RR \times J\big)\, .
\]
Repeating the proof of \cref{thm: main}, we show that the $\bk$-vector space
\[
W \coloneqq \left\{ \frac{P(\bt,f)}{E_y(\bt,f)} : 
P\in \bk[\bt,y],\, \NP(P) \subset C' \right\}
\]
contains $f$ and is invariant under~$S_{\br}$ for all $\br \in G$. Noticing that $\Delta_G$
commutes with $S_{\br}$ whenever $\br \in G$, we conclude that $\Delta_G(W)$ is invariant under
$S_{\br}$ for all $\br \in G$ as well.

Let $V$ be the $K_{0,G}$-span of $\Delta_G(W)$ in $K_G \coloneqq \text{Frac}(R_G)$.
By linearity, we find that $V$ is spanned by the elements $\bt^{\bi} f^j/E_y(\bt,f)$ for $(\bi, j)$
running over $C' \cap \NN^{n+1}$.
Besides, two fractions $\bt^{\bi} f^j/E_y(\bt,f)$ 
and $\bt^{\bi'} f^{j'}/E_y(\bt,f)$ are $K_{0,G}$-collinear
as soon as $\bi \equiv \bi' \bmod G$, which occurs if and
only if $\pi_G(\bi, j) = \pi_G(\bi', j)$. 
The dimension of $V$ over $K_{0,G}$ is then upper bounded by the cardinality of $\pi_G(C' \cap
\NN^{n+1})$, which is itself upper bounded by $N$ (see the last paragraph of the proof of~\cref{thm:
main} for more details).

The Frobenius map $\F$ acts as an endomorphism of $K_{0,G}$. 
We consider the ``relative'' Frobenius map of $K_G$ defined by 
\[\begin{array}{rcl}
\psi: \quad 
K_G \otimes_{K_{0,G},\F} K_{0,G} & \longrightarrow & K_G \\
x \otimes y & \mapsto & x^p y
\end{array}\]
where the notation $\otimes_{K_{0,G}, \F}$ means that we view
$K_{0,G}$ as an algebra over itself via $\F$.  Hence,
in $K_G \otimes_{K_{0,G}, \F} K_{0,G}$, we have 
$1 \otimes y = y^p \otimes 1$.
This construction ensures that $\psi$ is a $K_{0,G}$-linear
isomorphism. 
Moreover, it is related to the section operators \emph{via} the
formula 
\[\psi^{-1}(f) = \sum_{\br \in G_p} S_{\br}(f) \otimes \bt^{\br}\,,\]
where we let $G_p \subset G$ denote a set of representatives of $G/pG$.
Recall that we have proved earlier that $V$ is closed under the action of $S_{\br}$ for all $\br
\in G$. Therefore, we find that $\psi^{-1}$ induces a $K_{0,G}$-linear morphism from $V$ to $V
\otimes_{K_{0,G}, \F} K_{0,G}$.
Being the restriction of an injective map, this morphism is clearly injective.
Given that $V$ is finite dimensional over $K_{0,G}$ and that $\dim_{K_{0,G}} V = \dim_{K_{0,G}} (V
\otimes_{K_{0,G}, \F} K_{0,G})$, we conclude that it is an isomorphism.
Hence $\psi$ takes $V \otimes_{K_{0,G}, \F} K_{0,G}$ to $V$, which further implies that $V$
is invariant under the Frobenius map.

In particular, $\Delta_G(f)^{p^s}$ lies in $V$ for all nonnegative integers~$s$.
Since moreover $\dim_{K_{0,G}} V \leq N$, it follows that $\Delta_G(f), \Delta_G(f)^p, \ldots,
\Delta_G(f)^{p^N}$ must be linearly dependent over $K_{0,G}$.
Thus, there exist $c_0, c_1, \ldots, c_N \in K_{0,G}$, not all zero, such that
\[c_0 \cdot \Delta_G(f) +  c_1 \cdot \Delta_G(f)^p + \cdots + 
c_N \cdot \Delta_G(f)^{p^N} = 0 \,,\]
as desired.
\end{proof}

\subsection{Proof of \cref{thm:diag}}
We apply \cref{thm:gendiag} with the group $G$ generated
by $(1, \ldots, 1)$. As already noticed, the diagonal $\Delta_G$
is then the usual diagonal $\Delta$, up to the identification 
$t \coloneqq t_1 \cdots t_n$.
Let $A(\bt, y)$ be the minimal polynomial of $f$, so that $A$ has degree $d$, total height $h$, and partial height $\bh=(h_1,\ldots,h_n)$.
Let $\pi_G$ be the mapping and $C$ be the convex set defined
in the statement of \cref{thm:gendiag}. 
By \cref{thm:gendiag}, it remains to prove that $\Card \big(\pi_G(C \cap \NN^{n+1})\big) \leq N$. 

Let $c \coloneqq (a_1, \ldots, a_n, b) \in C \cap \NN^{n+1}$. We have $-1 < b \leq d$ and, given
that $b$ is an integer, we conclude that $0 \leq b \leq d$.
Moreover, up to translating $c$ by an element of $G$, one may assume that $0 \leq a_i \leq h_i$ for
all $i \in \{1, \ldots, n\}$ and that $\sum_{i=1}^n a_i\leq h$.
Define $a \coloneqq \min\{a_1, \ldots, a_n\}$ and, for all $i$, set $\tilde a_i \coloneqq a_i - a$.
Then  one of the first $n$ coordinates of $\tilde c \coloneqq (\tilde a_1, \ldots, \tilde a_n, b)$ 
vanishes. 
On the other hand, one has $0\leq \tilde a_i \leq h_i$ and $\sum_{i=1}^n \tilde a_i \leq h$. 
Furthermore, $\pi_G(c) = \pi_G(\tilde c)$. 
This ensures that any element of $\pi_G(C \cap \NN^{n+1})$ has a preimage in each of the sets 
\[
\mathcal E_1 \coloneqq \left\{ (a_1,\ldots,a_n,b) \in \mathbb N^{n+1} :  b\leq d,  \forall i, a_i 
\leq h_i, \exists i, a_i=0 \right\}
\]
and 
\[
\mathcal E_2 \coloneqq \left\{ (a_1,\ldots,a_n,b) \in \mathbb N^{n+1} : b\leq d, \sum_{i=1}^na_i \leq h, \exists i, a_i=0 \right\} \,.
\]
Since
\[\Card(\mathcal E_1) = (d+1)\cdot \left(\prod_{i=1}^n (h_i+1) - \prod_{i=1}^nh_i\right)\] 
and 
\[\Card (\mathcal E_2) =   (d+1)\cdot \left( \binom{n+h}{n}-\binom{h}{n} \right) \,, \] 
we have that $ \Card \big(\pi_G(C \cap \NN^{n+1})\big) \leq \min \{\Card(\mathcal E_1) , \Card(\mathcal E_2)\} =   N$, as wanted. 

\section{Hadamard product and other similar products}\label{sec: hadamard}

Let $k$ be a field and $\bt\coloneqq(t_1,\ldots,t_n)$ be a vector of indeterminates. Given two
multivariate power series $f(\bt) \coloneqq\sum_{\bi\in\mathbb N^n}a(\bi)\bt^{\bi}$ and
$g(\bt)\coloneqq\sum_{\bi\in \mathbb N^n}b(\bi)\bt^{\bi}$ in $k[[\bt]]$, their Hadamard product is
defined by
\[
f\odot g \coloneqq \sum_{\bi\in\mathbb N^n}a(\bi)b(\bi)\bt^{\bi}\in k[[\bt]] \,.
\]
The Hadamard product is intimately connected to diagonalization. Indeed, we have
\[
\Delta(f)= \varphi\left(f\odot \frac{1}{1-t_1\cdots t_n}\right)\,,
\]
where $\varphi$ is the map defined by $t_1\cdots t_n \mapsto t$, while 
\[
f \odot g = \Delta_G(f(\bt)g(\by))\,,
\]
where $G$ is the subgroup of $\mathbb Z^{2n}$ generated by the vectors $v_i$, $1\leq i\leq n$,
whose $i$-th and $(i+n)$-th entries are $1$ and the other entries are $0$.

As with diagonals, if $k$ has characteristic zero, the Hadamard product of two algebraic power
series in $\bk [[\bt]]$ is in general transcendental.
The situation is totally different if $k$ has characteristic $p$.
When $k$ is a finite field and $n=1$, Furstenberg \cite{Fu67} proved that the ring of algebraic
power series is closed under Hadamard product.
This was independently extended to arbitrary fields $k$ of characteristic~$p$ and positive integers
$n$, by Denef and Lipshitz \cite{DL}, Sharif and Woodcock \cite{SW}, and Harase \cite{Ha88}. 
Harase \cite{Ha89} also obtained a quantitative version of this result when $k$ is a perfect field:
the degree of algebraicity of $f\odot g$ over $k(\bt)$ is bounded by $p^{Ap^B}$ for some $A$ and
$B$ that are made explicit and depend polynomially on the degrees and the heights of $f$ and $g$.

The following theorem improves considerably Harase's bound.  

\begin{theorem}
\label{thm: hadamard}
Let $k$ be a field of characteristic $p$ and let $f,g\in k[[\bt]]$ be two algebraic power series of
degree $d$ and $d'$, total height $h$ and $h'$, and partial height $\bh\coloneqq (h_1,\ldots,h_n)$ 
and $\bh'\coloneqq (h'_1,\ldots,h'_n)$, respectively. 
Set 
\[
N_1 \coloneqq (d+1) \cdot \min \left\{ \prod_{i=1}^n (h_i+1) , \binom{n+h}{n} \right\}
\]
and 
\[
N_2 \coloneqq (d'+1) \cdot \min \left\{ \prod_{i=1}^n (h'_i+1) , \binom{n+h'}{n} \right\} \,.
\]
Then  
$f \odot g$ is an algebraic power series of degree at most $p^{N_1 N_2}-1$ over $k(\bt)$. 
\end{theorem} 

We will prove this theorem in \cref{ssec:proofs}.

\begin{remark} 
Again, we could also bound the total height $h$ of $f\odot g$ following the argument in
\cite[Proposition 5.2]{AB13}.
\end{remark}

\subsection{Hurwitz and Lamperti products}\label{sec: HL}

All along this section, we only consider univariate power series, that is the case $n=1$. 
Given a field $k$ and two power series $f(t) \coloneqq\sum_{i=0}^{\infty}a(i)t^{i}$ and
$g(t)\coloneqq\sum_{i=0}^{\infty}b(i)t^{i}$ in $k[[t]]$, their Hurwitz product is defined by
\[
f \circ_{\text{H}} g \coloneqq \sum_{i=0}^{\infty} \left( \sum_{k=0}^{i}  \binom{i}{k} a(k) b(i-k)\right)t^{i}\in k[[t]] \,,
\]
and their Lamperti product  by 
\[
f \circ_{\text{L}} g \coloneqq \sum_{i=0}^{\infty} \left( \sum_{j+k+\ell=i}  
	\frac{i!}{j!k!\ell!} \alpha^{j}\beta^k \gamma^\ell a(j+k) b(\ell+k)\right)t^{i}\in k[[t]] \,,
\]
where the parameters $\alpha$, $\beta$, and $\gamma$ belong to $k$. 
The Lamperti product generalizes both the Hadamard product (taking $\alpha=\beta=0$ and $\gamma=1$)
and the Hurwitz product (taking $\alpha=\beta=1$ and $\gamma=0$).
When $k$ has positive characteristic, Harase \cite{Ha88,Ha89} proved that the Lamperti product of
two algebraic power series $f$ and $g$ remains algebraic and he provided a doubly exponential bound
(\emph{i.e.}\ of the form $p^{Ap^B}$) for the degree of $f \circ_{\text{L}} g$ (and thus for $f
\circ_{\text{H}} g$).
Again, our approach leads to a simply exponential bound.

\begin{theorem}
\label{thm: HL}
Let $k$ be a field of characteristic $p$ and let $f,g\in k[[t]]$ be two algebraic power series 
of  degree $d$ and $d'$ and height $h$ and $h'$, respectively. 
Set 
\[
N \coloneqq (d+1)(d'+1)(h+1)(h'+1) \,.
\]
Then $f \circ_{\text{L}} g$ is an algebraic power series of degree at most $p^N-1$ over $k(t)$. 
In particular, the same result holds for $f\odot g$ and $f \circ_{\text{H}} g$. 
\end{theorem}

\subsection{Proof of \cref{thm: hadamard,thm: HL}}
\label{ssec:proofs}
As previously, we let $K$ denote the field of fractions of $k[[\bt]]$ and $S_{\br}$,
$\br\in\{0,\ldots,p-1\}^n$ be the section operators.
We first deduce from \cref{thm: main} and \cref{prop: descente} the following
general result.

\begin{proposition}\label{prop: product}
Let $k$ be a field of characteristic $p$ and let $\star$ be a bilinear product defined over
$k[[\bt]]$ such that for all $f,g\in k[[\bt]]$ one has
\begin{equation}\label{eq: bilinear}
S_{\br}(f\star g) \in \spn_{k} \left\{ S_{\bi}f \star S_{\bj} g : \bi,\bj\in\{0,\ldots,p-1\}^n\right\} \,.
\end{equation}
Let $f_1,f_2\in k[[\bt]]$ and let us assume that for all $i\in \{1,2\}$ there exists a $k$-vector
space $W_i\subset K$ of dimension at most $d_i$ containing $f_i$ and invariant by~$\Omega_n$.
Then $f_1 \star f_2$ is an algebraic power series of degree at most $p^{d_1d_2}-1$ over $k(\bt)$. 
\end{proposition}

\begin{proof}
Let us first assume that $k$ is a perfect field. Let $h_1,\ldots,h_r$ be a basis of $W_1$ and
$h'_1,\ldots,h'_s$ be a basis of $W_2$.
Set 
\[
W \coloneqq \spn_k \{ h_i \star h'_j : 1\leq i\leq r, \, 1\leq j\leq s \}  \,.
\]
Then $W$ has dimension at most $d_1d_2$ and the bilinearity of $\star$ implies that $W$ contains $f
_1\star f_2$. On the other hand, since $W_1$ and $W_2$ are invariant under $\Omega_n$, we infer
from the bilinearity of $\star$, the semilinearity of the section operators, and \eqref{eq:
bilinear} that $W$ is also invariant under $\Omega_n$. Then it follows classically that $f_1\star
f_2$ is algebraic over $k(\bt)$ with degree at most $p^{d_1d_2}-1$ (cf., for instance, \cite{Ha89}
and \cite[Proposition 5.1]{AB13}).

Now, if $k$ is an arbitrary field of characteristic $p$, one can pass to its perfect closure
$k_{p}$ and then apply the previous argument to obtain that $f_1\star f_2$ as degree at most
$p^{d_1d_2}-1$ over $k_{p}(\bt)$. By \cref{prop: descente}, we deduce that the degree of
$f_1\star f_2$ over $k(\bt)$ is also at most $p^{d_1d_2}-1$, as wanted. 
\end{proof}

\begin{proof}[Proof of \cref{thm: hadamard}]
After noticing that 
\[
S_{\br} (f\odot g)= S_{\br}(f) \odot S_{\br}(g) \,,
\]
the proof follows directly  from \cref{prop: product} and \cref{cor: main}. 
\end{proof}

\begin{proof}[Proof of \cref{thm: HL}]
After noticing, as in \cite{Ha89}, that 
\[
S_{r} (f \circ_{\text{L}} g)= \sum_{\substack{ 0\leq s,r \\  s+t\leq r}} 
\frac{r!}{s!t!(r-s-t)!}  \alpha^{s/p} \beta^{t/p} \gamma^{(r-s-t)/p} S_{r-t}(f) \circ_{\text{L}}  S_{r-s}(g) \,,
\]
the proof follows directly from \cref{prop: product} and \cref{cor: main} (with $n=1$). 
\end{proof}

\section{Algorithmic consequences of \cref{thm: main}}\label{sec: algo}

Throughout this section, we assume for simplicity that $k$ is a finite field~$\mathbb F_q$. We
address the question of the efficient computation of \emph{one} ``faraway'' coefficient of a power
series $f(\bt) \in \bk[[\bt]]$ assumed to be algebraic over $k(\bt)$.

It was pointed out in \cite[Corollary 4.5]{AlSh92} that the $M$-th term of an automatic sequence
(and more generally of a $k$-regular sequence) can be computed using $O(\log M)$ operations
in~$k$.
By Christol's theorem (and its
multivariate version) it follows that, given $\bi \coloneqq  (i_1, \ldots, i_n)$, the coefficient of
$\bt^{\bi} \coloneqq  t_1^{i_1} \cdots t_n^{i_n}$ in the expansion of $f$ can be computed in $O(\log M)$
operations in $k$, where $M \coloneqq \max(i_1, \ldots, i_n)$.
(For fields of characteristic zero, there is no algorithm achieving polynomial time in $\log M$ for the same task, even if $n=1$.)
However, this estimate is oversimplified in the sense that the $O(\cdot)$ hides
dependencies in the other parameters, namely the characteristic~$p$ of the ground field $k$,
the number of variables~$n$ and the various algebraicity degree and heights of~$f$.
The actual efficiency of any algorithm that computes the coefficient of $\bt^{\bi}$ of $f$ heavily
depends on these parameters, especially given that before entering the $O(\log M)$-part, some of
these algorithms may need to perform precomputations whose cost is so large with respect to the
other parameters, that the algorithms are highly inefficient in practice. This is particularly true
for the class of algorithms that start with building a $q$-automaton: the
running time of this precomputation depends on number of states of the automaton, which can be huge.
For these reasons, in the algorithmic design, it is important to care about the dependencies with
respect to all the other parameters.
This is the object of this section.

\subsection{The algorithm}

For algorithmic purposes, the starting point is always to find a suitable
finite representation of the objects we want to compute with. In our setting, it is of course not possible to represent a power series $f(\bt)$ by
its full sequence of coefficients in $\bk$, because this sequence is an infinite object. However,
when $f(\bt)$ is algebraic, we can hope to come back to a finite representation by working with an
annihilating polynomial of~$f(\bt)$, together with sufficiently many initial coefficients in the
expansion of~$f(\bt)$. This indeed works but requires some caution, given that the aforementioned
polynomial may in general have several roots.
In what follows, we shall encode an algebraic power series $f(\bt) \coloneqq  \sum_{\bi \in \NN^n} a(\bi)
\bt^{\bi}$ by the following data:
\begin{enumerate}[(a)]

\item its minimal polynomial $E(\bt, y) \in k[\bt, y]$ over $k(\bt)$, which we normalize (up to a
unit in $k$) by requiring it to have polynomial coefficients in $\bk[t]$ that are globally coprime,

\item a minimal element (for the product order on $\NN^n$), denoted by $\brho \coloneqq  (\rho_1, \ldots,
\rho_n)$, of $\NP\big(E_y(\bt, y)\big) \cap \NN^n$ (which is nonempty thanks to Lemma~\ref{lem:separable}),

\item the coefficients $a(\bi)$ for all tuples $\bi \coloneqq  (i_1, \ldots, i_n)$ with $0 \leq i_j \leq
\rho_j$ for all $1\leq j \leq n$.
\end{enumerate}

\begin{lemma}
\label{lem:encoding}
The previous data uniquely determines the algebraic power series~$f$.
\end{lemma}

\begin{proof}
Let $\bi \coloneqq  (i_1, \ldots, i_n) \in \NN^n$ be a multiindex.
For each $j \in \{1, \ldots, n\}$, we write the decomposition in base $p$ of $i_j$
\[i_j = \sum_{m=0}^{\ell-1} r_{j,m}\: p^m \,, \]
where $\ell$ is a positive integer and all the $r_{j,m}$'s are integers between $0$ and $p{-}1$.
For each $m$, we form the tuple $\br_m \coloneqq  (r_{1,m}, \ldots, r_{n,m})$ and consider the corresponding
section operator $S_{\br_m}$.
It follows from the definitions that the coefficient $a(\bi)$ is equal to the $p^\ell$-th power of 
the constant coefficient of the power series
\[g(\bt) \coloneqq 
S_{\br_{\ell-1}} \circ \cdots \circ S_{\br_1} \circ
S_{\br_0}\big(f(\bt)\big)\,.\]
Moreover, by Eq. \eqref{Sr_P_Q},
there exists a polynomial $Q \in k[\bt, y]$ such that
\[g(\bt) = \frac{Q(\bt, f(\bt))}{E_y(\bt, f(\bt))}\,,
\quad\textit{i.e.}\quad
Q(\bt, f(\bt)) = g(\bt) \cdot E_y(\bt, f(\bt))\,.\]
Identifying the coefficients in $\bt^{\brho}$ in the latter equality, we find
\[[\bt^{\brho}] Q(\bt, f(\bt))
= \sum_{\bu + \bv = \brho} [\bt^{\bu}] 
  g(\bt) \cdot [\bt^{\bv}] E_y(\bt, f(\bt)) \, ,\]
where the notation $[\bt^{\bj}]\varphi(\bt)$ refers to the coefficients in front of $\bt^{\bj}$ in
the power series $\varphi(\bt)$. On the other hand, we derive from the definition of $\brho$ (and
especially from the minimality condition (b)) that the unique $\bv \leq \brho$ for which the
coefficient $[\bt^{\bv}] E_y(\bt, f(\bt))$ does not vanish is $\brho$ itself. Therefore, we
conclude that
\[[\bt^{\brho}] Q(\bt, f(\bt))
= [\bt^0] g(\bt) \cdot [\bt^{\brho}] E_y(\bt, f(\bt))\]
which gives
\[a(\bi) = \big([\bt^0] g(\bt)\big)^{p^\ell} 
  = \left(\frac{[\bt^{\brho}] Q(\bt, f(\bt))}
    {[\bt^{\brho}] E_y(\bt, f(\bt))}\right)^{p^\ell}\,.\]
Since $f(\bt)$ is given at precision $O(t_1^{\rho_1+1} \cdots t_n^{\rho_n+1})$, we can compute
$Q(\bt, f(\bt))$ at the same precision; this ensures that the coefficient $[\bt^{\brho}] Q(\bt,
f(\bt))$ can be recovered from the set of data that we have at our disposal. 
Hence the same holds for $a(\bi)$.
Since~$\bi$ was chosen arbitrarily at the beginning of the proof, the lemma is proved.
\end{proof}

Importantly, we notice that the proof of \cref{lem:encoding} together with the explicit
formula~\eqref{eq:formulaQ} translate immediately to Algorithm~\ref{algo:ith} which computes the
coefficient~$a(\bi)$ of the power series $f(\bt)$.
Note that Algorithm~\ref{algo:ith} can be seen as a multivariate version of the algorithms in
\cite[Section 3]{BCCD}.

\begin{algo}
  \caption{$\bi$-th coefficient of an algebraic power series.\label{algo:ith}}
  \begin{algoenv}
    {A multiindex $\bi = (i_1, \ldots, i_n)$ and an algebraic power series $f(\bt) \in \bk[[\bt]]$ encoded by $(E(\bt,y), \brho, f(\bt) \bmod \bt^{\brho + 1})$
     }
    {The $\bi$-th coefficient of $f(\bt)$}
    \State 1. For $j = 1, \ldots, n$,\\
      \hphantom{1. }write the decomposition of $i_j$ in base $p$: $i_j = \sum_{m=0}^{\ell-1} r_{j,m} \: p^m$
    \State 2. Set $Q_0(\bt, y) \coloneqq y \cdot E_y(\bt, y)$
    \State 3. For $i = 0, 1, \ldots, \ell - 1$,\\
      \hphantom{4. }set $Q_{i+1}(\bt, y) \coloneqq S_{f, r_{1,m}, \ldots, r_{n,m}, p-1}\big(Q_i(\bt, y) \cdot E(\bt,y)^{p-1}\big)$
    \State 4. Compute $Q_\ell(\bt, f(\bt)) \bmod \bt^{\brho + 1}$ and set $\alpha \coloneqq [\bt^{\brho}] Q_\ell(\bt, f(\bt))$
    \State 5. Compute $E_y(\bt, f(\bt)) \bmod \bt^{\brho + 1}$ and set $\beta \coloneqq [\bt^{\brho}] E_y(\bt, f(\bt))$
    \State 6. Return $(\alpha/\beta)^{p^\ell}$
  \end{algoenv}
\end{algo}

\smallskip
We now study the (arithmetic) complexity of Algorithm~\ref{algo:ith}. In what follows, we will
express complexity of algorithms in terms of the number of operations they perform in the ground
field~$k$. By ``operation'', we mean either a classical arithmetical (field) operation (addition,
subtraction, multiplication, division) or an application of the Frobenius map $\F$, or its inverse.
We recall the soft-$O$ notation $\softO$: by definition, $\softO(c)$ is the union of the $O(c
\log^k(c))$ for $k$ varying in~$\NN$. Using~$\softO$ instead of the more customary $O$-notation
makes it possible to ``hide'' logarithmic factors.

\begin{theorem}
\label{thm:complexity}
Let $f(\bt) \in k[[\bt]]$ be an algebraic power series encoded by the data 
\[(E(\bt, y), \brho, f(\bt) \bmod \bt^{\brho + 1}) \,.\]
Let $\bh \colonequal (h_1, \ldots, h_n)$ be the vector of partial heights of $E(\bt, y)$ and $d$ be its
degree.
On input $f(\bt)$ and $\bi \colonequal (i_1, \ldots, i_n)$, Algorithm~\ref{algo:ith} performs at most
\[\softO\big(
  2^n d p^{n+1} (h_1{+}1) \cdots (h_n{+}1) \log M
  + 2^n (\rho_1{+}1) \cdots (\rho_n{+}1)\big)\]
operations in $k$, where $M \colonequal \max(i_1, \ldots, i_n)$.
\end{theorem}

\begin{proof} 
By construction, all intermediate polynomials $Q_i(\bt,y)$ have partial heights at most~$\bh$ and
degree at most~$d$.
Recall that polynomials in $k[\bt,y]$ of degree at most $d$ and partial heights at most $\bh =
(h_1, \ldots, h_n)$ can be multiplied in $\softO \left( 2^n d (h_1+1) \cdots (h_n+1) \right)$
operations in $k$, using Fast Fourier Transform (FFT) multiplication of univariate
polynomials~\cite{CaKa91} and Kronecker's substitution~\cite{Pan94}.
Therefore, each iteration of the loop in line~3 requires at most
$\softO\big(2^n d p^{n+1} (h_1{+}1) \cdots (h_n{+}1)\big)$ operations in~$k$.
Besides, the number of times this loop is executed, namely $\ell$, grows at most logarithmically
with respect to~$M$.
The total cost of line 3 then stays within
\[\softO\big(
  2^n d p^{n+1} (h_1{+}1) \cdots (h_n{+}1) \log M\big)\]
operations in $k$.
Similarly, the computations in lines~4 and~5 require at most $\softO\big(2^n (\rho_1{+}1) \cdots
(\rho_n{+}1)\big)$ operations in~$k$. 
Adding both contributions, we find the announced complexity.
\end{proof}

\begin{remark}
A small optimization can be applied to Algorithm~\ref{algo:ith}.
It consists in computing $Q_i^{\F^i}$ (that is the polynomial obtained from $Q_i$ by applying $\F^i$
to each of its coefficients) instead of $Q_i$ on line 3, thanks to the recurrence relation:
\[Q_{i+1}^{F^{i+1}}(\bt, y) = 
  S^\F_{f, r_{1,m}, \ldots, r_{n,m}, p-1}\big(Q_i^{F^i}(\bt, y) \cdot E^{\F^i}(\bt,y)^{p-1}\big)\,, \]
where $S^\F_{f,\br,s}$ is the same operator as $S_{f,\br,s}$ except that we do not take preimages
of the coefficients by the Frobenius map $\F$.
Proceeding this way, we retrieve $\F^\ell(\alpha) = \alpha^{p^\ell}$ by selecting the
coefficient of $\bt^{\brho}$ in $Q_\ell^{\F^\ell} (\bt, f(\bt))$ and can return
$\F^\ell(\alpha)/\beta^{p^\ell}$ on line~6.
With this optimization, it becomes unnecessary to apply inverses of~$\F$ in 
Algorithm~\ref{algo:ith}. 
This may be beneficial since, in practice, applying $\F^{-1}$ may be a more expensive operation 
than applying $\F$.
\end{remark}

\begin{remark}
It is possible to give \emph{a priori} bounds on $\brho$ in terms of $\bh$ and~$d$.
Indeed, let $r(\bt)$ be the resultant in $y$ of the polynomials $E(\bt, y)$ and $E_y(\bt, y)$. A
calculation shows that the $t_i$-degree of $r(\bt)$ is upper bounded by $2 d h_i$.
Besides, it follows from the standard properties of resultants that $E_y(\bt, f(\bt))$ divides
$r(\bt)$ in the ring $k[[\bt]]$.
Hence the Newton polytope of $E_y(\bt, f(\bt))$ necessarily has one point in the box $[0, 2 d h_1]
\times \cdots \times [0, 2 d h_n]$, from which we derive that one can always choose $\brho \leq 2 d
{\cdot} \bh$.
The above discussion shows that the complexity of Algorithm~\ref{algo:ith} can be controlled in
terms of $\bh$, $d$ and $\log M$ only: to simplify notation, setting 
\[H \colonequal (h_1{+}1) \cdots (h_n{+}1)\,,\] 
the number of operations used by Algorithm~\ref{algo:ith} is at most
\begin{equation}\label{complexity}
	\softO\big(2^n d p^{n+1} H \log M + 4^n d^n H \big)\,.
\end{equation}
\end{remark}

For $n=1$ and $h_1= h >0$, the cost \eqref{complexity} reads $\softO\big(d h p^{2} \log M
\big)$.
This is similar to the complexity estimate $\softO\big( h (d+h) p^{2} + h^2 (d+h)^2 \log M \big)$
of~\cite[Algorithm 3]{BoChDu16}.
Faster algorithms are available when~$n=1$, with quasi-linear complexity in~$p$:
$\softO\big(h^2(h+d)^2 \log M + h(h+d)^5 p\big)$ (\cite[Theorem~11]{BoChDu16}) and even
$\softO\big(d^2 h^2 \log M + d^2 h p + d^3 h \big)$ (\cite[Theorem~3.4]{BCCD}).

For a general $n$, the complexity estimate \eqref{complexity} is exponentially better (with respect
to~$p$) than all estimates that could be deduced from known approaches. This important improvement
is of course ultimately inherited from~\cref{thm: main}.
It would be interesting to design faster variants of Algorithm~\ref{algo:ith}, whose complexities
improve the estimate \eqref{complexity}, e.g. replacing the term $p^{n+1}$ by $p^{n}$.

\subsection{Applications to diagonals}

Assume that $f(\bt)\in \mathbb F_p[[\bt]]$ is an algebraic power 
series in $n$ variables of
degree $d$ and partial height $\bh\coloneqq (h_1,\ldots,h_n)$.
We consider here the following algorithmic problem:
given $M\in\N$, how fast can one compute the $M$-th coefficient
in the expansion of the diagonal $\Delta(f)$?

When $f$ is rational, Rowland and Yassawi proposed in \cite[Section 2]{RoYa15} two algorithms to
compute finite automata that could be used to compute terms of the coefficient sequence of
$\Delta(f)$. However, the number of states of the produced automata is prohibitively
large: Remark 2.2 in \cite{RoYa15} provides an upper bound of the form $p^{(h+1)^n}$, where $h$ is
the total height of~$f$.

By~\cref{thm: modp} and the discussion above, $\Delta(f)$ is an algebraic power series in $\mathbb F_p[[t]]$ of
degree $D_\Delta \leq p^N-1$ and height $H_\Delta\leq Np^N$ where $N$ is at
most $(d+1)H$ and, as before, $H = (h_1{+}1) \cdots (h_n{+}1)$.
Hence, a ``naive'' way of computing the coefficient $[t^M] \Delta(f)$ would be to
first determine an annihilating polynomial $E\in \mathbb F_p[t,y]$ for $\Delta(f)$, and
then to apply the algorithms of~\cite{BoChDu16} or \cite{BCCD} to this $E$.
Starting from the first $2(D_\Delta{+}1)(H_\Delta{+}1)$ terms in the 
expansion of $\Delta(f)$, one can compute such an~$E$, by (structured)
linear algebra, using
$\tilde{O}(D_\Delta^{\omega} H_\Delta) \subseteq \tilde{O}(N p^{(\omega+1)N})$
operations in~$\mathbb{F}_p$, as explained in \cite[p.~128]{BCCD}.
Here $\omega \in [2,3]$ is a feasible exponent for the matrix
multiplication.
Hence, using~\cite[Theorem~3.4]{BCCD}, we conclude that, for $M\gg 0$, 
the coefficient $[t^M] \Delta(f)$ can be computed in 
\begin{multline*}
\softO\big(H_\Delta^2(H_\Delta+D_\Delta)^2 \log M + D_\Delta^2 H_\Delta p + D_\Delta^3 H_\Delta \big) \\
\subseteq \softO\big(d^4 H^4 p^{4(d+1)H} \log M\big)
\end{multline*}
operations in $\mathbb{F}_p$.
This complexity estimate is quasi-optimal with respect to~$M$, but its dependence in~$p$ is far from optimal.

A much better method is to apply Algorithm~\ref{algo:ith} with $\bi = (M, \ldots, M)$ directly,
without first precomputing an annihilating polynomial for $\Delta(f)$. The resulting
complexity is then provided by \cref{thm:complexity}, and is bounded in terms of $n, p, d, h_1,
\ldots, h_n$ by the estimate in~\eqref{complexity}, namely
\[
	\softO\big(2^n d p^{n+1} H \log M + 4^n d^n H \big)\,.
\]
This method then allows to compute the $M$-th coefficient of the 
diagonal of~$f(t_1,\ldots,t_n)$ in arithmetic complexity linear 
in~$\log M$ and quasi-linear in $p^{n+1}$. This result was previously 
known only in the bivariate case. For more than two variables, 
previous algorithms with complexity linear in~$\log M$ required (at 
least) doubly exponential time in the arithmetic size of~$f$.

Note finally that the truth of Christol's conjecture~\cite[Conjecture~4]{Ch90} would imply that
$\log(M)$-time algorithms for computing the $M$-th term modulo~$p$ might well exist for the whole
class of integer sequences with geometric growth which are P-finite (i.e., which satisfy linear
recurrence relations with polynomial coefficients in the index~$n$).

\subsection{Examples}

We consider the Ap\'ery numbers $A(n) \colonequal \sum_{k=0}^n
\binom{n}{k}^2\binom{n+k}{k}^2 $. 
Their generating function $\sum_{n=0}^\infty A(n) t^n$ is the
diagonal of the rational function in $n=4$ variables~\cite{St14}
\[ f(t_1, t_2, t_3, t_4) \colonequal \frac{1}{(1-t_1-t_2)(1-t_3-t_4) - t_1t_2t_3t_4} \,\cdot\]
Hence, by~\eqref{complexity}, $A(M) \bmod p$ can be computed in $O (p^{5} \log M)$ operations
in~$\mathbb{F}_p$. This estimate can be lowered to $\softO (p^{4} \log M)$ by using that $\sum_{n
\geq 0} A(n) t^n$ is also the diagonal of an algebraic function in $n=3$ variables.
In this particular case, a better complexity bound 
can be obtained by exploiting nontrivial arithmetic properties of the 
Apéry numbers.
Indeed, it turns out that the sequence $(A(n))_{n\geq 0}$ is
$p$-Lucas~\cite{Ge82}: if $M=(i_{\ell-1} \ldots i_1 i_0)_p$ is the base-$p$ expansion of $M$, then
$A(M) = A({i_{\ell-1}}) \cdots A({i_1}) A({i_0}) \bmod p$, hence it is sufficient to precompute $A(0) \bmod p,
\ldots, A({p{-}1}) \bmod p$. This can be done for a cost of $\softO(p^{4})$ 
operations in~$\mathbb{F}_p$; after this, computing $A(M) \bmod p$ only requires $O(\log M)$ extra operations in~$\mathbb{F}_p$.

However, beyond this example, many other interesting integer sequences 
are not $p$-Lucas, but are known to admit diagonals of algebraic 
functions as generating functions.
Most of the integer sequences from the combinatorial literature, which are P-finite and have a
geometric growth, are known to fall into this class.
For instance, the sequence $1, 6, 222, 9918, \ldots$ (\href{http://oeis.org/A144045}{A144045})
counting diagonal rook paths on a 3D chessboard was proved in~\cite{BCHP} to satisfy the recurrence
\begin{multline*}
2n^2(n-1)a(n) -(n-1)(121n^2-91n-6)a(n-1) \\
- (n-2) (475 n^2-2512n+2829)a(n-2)
  + 18(n-3)(97n^2-519n+702)a(n-3) \\
- 1152 (n-3) (n-4)^2a(n-4)= 0\,,  \qquad \text{for~$n\geq 4$}\,,
\end{multline*}      
and to admit a generating function 
$F(t) \coloneqq  \sum_{n \geq 0} a(n) t^n$ equal to 
\[
F(t)  =  1 + 6 \cdot \bigintsss_0^t \frac{ \,\pFq21{1/3}{2/3}{2}
  {\displaystyle \frac{27 w(2-3w)}{(1-4w)^3}}}{(1-4w)(1-64w)} \, dw \,.
\]  
Given a prime~$p$, neither the recurrence, nor the closed form of $F(t)$ are well-suited to compute
rapidly the value $a(M) \bmod p$ for high values of~$M$.
In exchange, $a(M)$ is the $M$-th coefficient of the diagonal of the rational function in $3$ 
variables
\[
f(t_1, t_2, t_3) \colonequal
\frac{(1-{t_1})(1-{t_2})(1-{t_3})}{1-2({t_1}+{t_2}+{t_3})+3({t_1}{t_2}+{t_2}{t_3}+{t_1}{t_3})-4{t_1}{t_2}{t_3}} \, \cdot \] 
Hence, by~\eqref{complexity}, $a(M) \bmod p$ can be computed using $\softO (p^4  \log M)$
operations in $\mathbb{F}_p$.

A similar example is given by the sequence $1, 2, 18, 255, 4522, \ldots$
(\href{http://oeis.org/A151362}{A151362}) whose $n$-th term $q(n)$ counts walks in the quarter
plane $\N^2$ of length~$2n$ starting at the origin and using steps in the set $\{$N, S, NE, SE, NW,
SW$\}$.
The generating function $Q(t) \coloneqq  \sum_{n \geq 0} q(n) t^{2n}$ of this sequence is known to be transcendental over $\mathbb{Q}(t)$ and to admit the $_2 F _1$ expression~\cite{BCHKP}
\[
Q(t) = 
{\frac {2}{{t}^{2}}\bigintsss_{0}^t \!\!\!\bigintsss_{0}^y \,{\frac {1}{ \left( 12\,{z}^{2}+1 \right) ^{3/2}}
\cdot \pFq21{3/4}{5/4}{2}
  { \frac {64 \,{z}^{2}}{ \left( 12\,{z}^{2}+1 \right) ^{2}}}}
  \,{\rm d}z\,{\rm d}y}
\]
and also the diagonal expression~\cite{MM}
 \[
Q(t) = 
 \Delta \left(
{\frac { \left( t_2^2-1 \right)  \left( t_3^2-1 \right) }{1-t_1 \left( t_2^2t_3^2+ t_2t_3^2+t_2^2+t_3^2+t_2
+1 \right) }}
\right)\,,
 \]
with $n=3$ and $(h_1,h_2,h_3) = (1,2,2)$. 
Using the last expression, $q(M) \bmod p$ can be computed in 
 $\softO (p^{4}  \log M)$
operations in~$\mathbb{F}_p$.
The same remark actually applies to all $19 \times 4 - 3$ transcendental generating functions of the form $Q(0,0), Q(1,0), Q(0,1)$ and $Q(1,1)$ from~\cite[Theorem 2]{BCHKP}.

\bibliographystyle{alpha}

\begin{thebibliography}{CKMFR80}

\bibitem[AB12]{AB12}
B.~Adamczewski and J.~Bell.
\newblock On vanishing coefficients of algebraic power series over fields of
  positive characteristic.
\newblock {\em Invent. Math.}, 187:343--393, 2012.

\bibitem[AB13]{AB13}
B.~Adamczewski and J.~Bell.
\newblock Diagonalization and rationalization of algebraic {L}aurent series.
\newblock {\em Ann. Sci. \'{E}c. Norm. Sup\'{e}r.}, 46:963--1004, 2013.

\bibitem[AB21]{AB_chapter}
B.~Adamczewski and J.~Bell.
\newblock Automata in number theory.
\newblock In {\em Handbook of automata theory. {V}ol. {II}. {A}utomata in
  mathematics and selected applications}, pages 913--945. EMS Press, Berlin,
  2021.

\bibitem[ABD19]{ABD}
B.~Adamczewski, J.. Bell, and E.~Delaygue.
\newblock Algebraic independence of {$G$}-functions and congruences ``\`a la
  {L}ucas''.
\newblock {\em Ann. Sci. \'{E}c. Norm. Sup\'{e}r.}, 52:515--559, 2019.

\bibitem[AGBS98]{AGBS}
J.-P. Allouche, D.~Gouyou-Beauchamps, and G.~Skordev.
\newblock Transcendence of binomial and {L}ucas' formal power series.
\newblock {\em J. Algebra}, 210:577--592, 1998.

\bibitem[AS92]{AlSh92}
Jean-Paul Allouche and Jeffrey Shallit.
\newblock The ring of {$k$}-regular sequences.
\newblock {\em Theoret. Comput. Sci.}, 98(2):163--197, 1992.

\bibitem[AS03]{AS03}
J.-P. Allouche and J.~Shallit.
\newblock {\em Automatic sequences}.
\newblock Cambridge University Press, Cambridge, 2003.

\bibitem[AY19]{AY}
B.~Adamczewski and R.~Yassawi.
\newblock A note on {C}hristol's theorem.
\newblock arXiv:1906.08703 [math.NT], \url{https://arxiv.org/abs/1906.08703},
  2019.

\bibitem[Bak93]{Baker1893}
H.~F. Baker.
\newblock Examples of the application of {N}ewton’s polygon to the theory of
  singular points of algebraic functions.
\newblock {\em Trans. Camb. Phil. Soc.}, XV(IV):403--450, 1893.

\bibitem[BCCD19]{BCCD}
A.~Bostan, X.~Caruso, G.~Christol, and P.~Dumas.
\newblock Fast coefficient computation for algebraic power series in positive
  characteristic.
\newblock In {\em Proceedings of the {T}hirteenth {A}lgorithmic {N}umber
  {T}heory {S}ymposium}, volume~2 of {\em Open Book Ser.}, pages 119--135.
  Math. Sci. Publ., Berkeley, CA, 2019.

\bibitem[BCD16]{BoChDu16}
Alin Bostan, Gilles Christol, and Philippe Dumas.
\newblock Fast computation of the {$N$}th term of an algebraic series over a
  finite prime field.
\newblock In {\em Proceedings of the 2016 {ACM} {I}nternational {S}ymposium on
  {S}ymbolic and {A}lgebraic {C}omputation}, pages 119--126. ACM, New York,
  2016.

\bibitem[BCvH{\etalchar{+}}17]{BCHKP}
A.~Bostan, F.~Chyzak, M.~van Hoeij, M.~Kauers, and L.~Pech.
\newblock Hypergeometric expressions for generating functions of walks with
  small steps in the quarter plane.
\newblock {\em European J. Combin.}, 61:242--275, 2017.

\bibitem[BCvHP12]{BCHP}
A.~Bostan, F.~Chyzak, M.~van Hoeij, and L.~Pech.
\newblock Explicit formula for the generating series of diagonal 3{D} rook
  paths.
\newblock {\em S\'{e}m. Lothar. Combin.}, 66:Art. B66a, 27, 2011/12.

\bibitem[Bri17]{Br17}
A.~Bridy.
\newblock Automatic sequences and curves over finite fields.
\newblock {\em Algebra Number Theory}, 11:685--712, 2017.

\bibitem[Chr79]{Ch79}
G.~Christol.
\newblock Ensembles presque periodiques {$k$}-reconnaissables.
\newblock {\em Theoret. Comput. Sci.}, 9:141--145, 1979.

\bibitem[Chr90]{Ch90}
G.~Christol.
\newblock Globally bounded solutions of differential equations.
\newblock In {\em Analytic number theory ({T}okyo, 1988)}, volume 1434 of {\em
  Lecture Notes in Math.}, pages 45--64. Springer, Berlin, 1990.

\bibitem[Chr15]{Ch15}
G.~Christol.
\newblock Diagonals of rational fractions.
\newblock {\em Eur. Math. Soc. Newsl.}, 97:37--43, 2015.

\bibitem[CK91]{CaKa91}
David~G. Cantor and Erich Kaltofen.
\newblock On fast multiplication of polynomials over arbitrary algebras.
\newblock {\em Acta Inform.}, 28(7):693--701, 1991.

\bibitem[CKMFR80]{CKMR}
G.~Christol, T.~Kamae, M.~Mend\`es~France, and G.~Rauzy.
\newblock Suites alg\'{e}briques, automates et substitutions.
\newblock {\em Bull. Soc. Math. France}, 108:401--419, 1980.

\bibitem[Del84]{De84}
P.~Deligne.
\newblock Int\'{e}gration sur un cycle \'{e}vanescent.
\newblock {\em Invent. Math.}, 76:129--143, 1984.

\bibitem[Der07]{De07}
H.~Derksen.
\newblock A {S}kolem-{M}ahler-{L}ech theorem in positive characteristic and
  finite automata.
\newblock {\em Invent. Math.}, 168:175--224, 2007.

\bibitem[DL87]{DL}
J.~Denef and L.~Lipshitz.
\newblock Algebraic power series and diagonals.
\newblock {\em J. Number Theory}, 26:46--67, 1987.

\bibitem[FKdM00]{FKdM}
J.~Fresnel, M.~Koskas, and B.~de~Mathan.
\newblock Automata and transcendence in positive characteristic.
\newblock {\em J. Number Theory}, 80:1--24, 2000.

\bibitem[Fur67]{Fu67}
H.~Furstenberg.
\newblock Algebraic functions over finite fields.
\newblock {\em J. Algebra}, 7:271--277, 1967.

\bibitem[Ges82]{Ge82}
I.~Gessel.
\newblock Some congruences for {A}p\'{e}ry numbers.
\newblock {\em J. Number Theory}, 14:362--368, 1982.

\bibitem[Har88]{Ha88}
T.~Harase.
\newblock Algebraic elements in formal power series rings.
\newblock {\em Israel J. Math.}, 63:281--288, 1988.

\bibitem[Har89]{Ha89}
T.~Harase.
\newblock Algebraic elements in formal power series rings. {II}.
\newblock {\em Israel J. Math.}, 67:62--66, 1989.

\bibitem[Hod29]{Hodge29}
W.~V.~D. Hodge.
\newblock The {I}solated {S}ingularities of an {A}lgebraic {S}urface.
\newblock {\em Proc. London Math. Soc. (2)}, 30(2):133--143, 1929.

\bibitem[Hov78]{Khovanskii78}
A.~G. Hovanski\u{\i}.
\newblock Newton polyhedra, and the genus of complete intersections.
\newblock {\em Funktsional. Anal. i Prilozhen.}, 12(1):51--61, 1978.

\bibitem[MM16]{MM}
S.~Melczer and M.~Mishna.
\newblock Asymptotic lattice path enumeration using diagonals.
\newblock {\em Algorithmica}, 75:782--811, 2016.

\bibitem[Pan94]{Pan94}
Victor~Y. Pan.
\newblock Simple multivariate polynomial multiplication.
\newblock {\em J. Symbolic Comput.}, 18(3):183--186, 1994.

\bibitem[RY15]{RoYa15}
E.~Rowland and R.~Yassawi.
\newblock Automatic congruences for diagonals of rational functions.
\newblock {\em J. Th\'{e}or. Nombres Bordeaux}, 27:245--288, 2015.

\bibitem[Sal86]{Sa86}
O.~Salon.
\newblock Suites automatiques \`a multi-indices.
\newblock {\em S\'em. Th\'eorie Nombres Bordeaux (1986-1987),}, Expos\'e
  4:1--27, 1986.

\bibitem[Sal87]{Sa87}
O.~Salon.
\newblock Suites automatiques \`a multi-indices et alg\'{e}bricit\'{e}.
\newblock {\em C. R. Acad. Sci. Paris S\'{e}r. I Math.}, 305:501--504, 1987.

\bibitem[Str14]{St14}
A.~Straub.
\newblock Multivariate {A}p\'{e}ry numbers and supercongruences of rational
  functions.
\newblock {\em Algebra Number Theory}, 8:1985--2007, 2014.

\bibitem[SW88]{SW}
H.~Sharif and C.~F. Woodcock.
\newblock Algebraic functions over a field of positive characteristic and
  {H}adamard products.
\newblock {\em J. London Math. Soc. (2)}, 37:395--403, 1988.

\bibitem[vdP93]{vdP90}
A.~J. van~der Poorten.
\newblock Power series representing algebraic functions.
\newblock In {\em S\'{e}minaire de {T}h\'{e}orie des {N}ombres, {P}aris,
  1990--91}, volume 108 of {\em Progr. Math.}, pages 241--262. Birkh\"{a}user
  Boston, Boston, MA, 1993.

\bibitem[VM21]{VM21}
D.~Vargas-Montoya.
\newblock Alg\'{e}bricit\'{e} modulo {$p$}, s\'{e}ries
  hyperg\'{e}om\'{e}triques et structures de {F}robenius fortes.
\newblock {\em Bull. Soc. Math. France}, 149:439--477, 2021.

\bibitem[VM23]{VM23}
D.~Vargas-Montoya.
\newblock {Monodromie unipotente maximale, congruences ``\`{a} la {L}ucas" et
  ind\'{e}pendance alg\'{e}brique}.
\newblock arXiv:2103.15192 [math.NT], to appear in \emph{Trans. Amer. Math.
  Soc.}, DOI: \url{https://doi.org/10.1090/tran/8913}, 2023.

\bibitem[WS89]{SW89}
C.~F. Woodcock and H.~Sharif.
\newblock On the transcendence of certain series.
\newblock {\em J. Algebra}, 121:364--369, 1989.

\bibitem[Zan09]{Za09}
U.~Zannier.
\newblock {\em Lecture notes on {D}iophantine analysis}, volume~8 of {\em
  Appunti. Scuola Normale Superiore di Pisa (Nuova Serie) [Lecture Notes.
  Scuola Normale Superiore di Pisa (New Series)]}.
\newblock Edizioni della Normale, Pisa, 2009.

\end{thebibliography}

\newcommand{\etalchar}[1]{$^{#1}$}

\end{document}